\documentclass[bj]{imsart}
\RequirePackage[OT1]{fontenc}
\RequirePackage{amsthm,amsmath,amsfonts,amssymb,natbib}
\RequirePackage[colorlinks,citecolor=blue,urlcolor=blue]{hyperref}
\arxiv{arXiv:0000.0000}
\startlocaldefs
\numberwithin{equation}{section}
\theoremstyle{plain}
\newtheorem{definition}{Definition}[subsection]

\newtheorem{example}{Example}[subsection]

\newtheorem{theorem}{Theorem}[section]
\newtheorem{lemma}{Lemma}[section]
\newtheorem{proposition}{Proposition}[section]
\newtheorem{corollary}{Corollary}[section]
\setcitestyle{citesep={,}}
\theoremstyle{remark}

\endlocaldefs
\begin{document}
\begin{frontmatter}
\title{L\'{e}vy copulas: a probabilistic point of view}
\runtitle{L\'{e}vy copulas: a probabilistic point of view}
\begin{aug}
\author{\fnms{Ayi} \snm{Ajavon}\thanksref{a}\ead[label=a,mark]{Ayi.Ajavon@canada.ca, ayiajav@yahoo.com }}
\address[a]{100 Tunney's Pasture Driveway, Ottawa (Ontario)  K1A 0T6
\printead{a}}
\affiliation{Statistics Canada}
\end{aug}
\begin{abstract}  
There is a one-to-one correspondence between L\'{e}vy copulas and proper copulas. The correspondence relies on a relationship between L\'{e}vy copulas sitting on $[0,+\infty]^d$ and max-id distributions. The max-id distributions are defined with respect to a partial order that is compatible with the inclusion of sets bounded away from the origin. An important consequence of the result is the possibility to define parametric L\'{e}vy copulas as mirror images of proper parametric copulas. For example, proper Archimedean copulas are generated by functions that are Williamson $d-$transforms of the cdf of the radial component of random vectors with exchangeable distributions $F_{R}$. In contrast, the generators of Archimedean L\'{e}vy copulas are Williamson $d-$transforms of $-\log(1-F_{R})$.
\end{abstract}

\begin{keyword}
\kwd{ Archimedean copulas, L\'{e}vy copulas,  max-infinite divisibility}
\end{keyword}
\end{frontmatter}
\section{Introduction}
L\'{e}vy processes are  processes with stationary  and independent increments in disjoint time intervals; they encompass the Brownian motion and the Poisson process, two important processes in Probability.  For a L\'{e}vy process $X=\{X_{t} \in \mathbb{R}^{d}: t \geq  0\}$, the random vector $X_{t}$ has an infinitely divisible distribution with the characteristic function $\rho$,
\begin{eqnarray*}
\label{LevyKhin}
\rho(y)=\exp \left\{i \langle y,\beta \rangle - \langle \gamma y, y \rangle +\int\{ e^{i  \langle y, x  \rangle}-1 -i \langle y, x \rangle \}1_{\{ \lVert x \rVert \leq 1\}}(x)\nu(dx) \right\}, \forall{y} \in \mathbb{R}^{d},
\end{eqnarray*}
where $<\cdot, \cdot>$ stands for the scalar product,   $\beta$  is a constant vector, $<\gamma  \cdot, \cdot>$ is a positive semi-definite quadratic form, $\nu$ is a L\'{e}vy measure on $\mathbb{R}^{d}$ satisfying $\nu((0,\ldots,0))=0$, $\int_{\mathbb{R}^{d}}\min(\lVert x \rVert^{2},1) \nu(dx)< \infty$. While the dependence in the Gaussian part of the L\'{e}vy process is characterized by the covariance matrix $\gamma$, the dependence in the pure jump part  is defined by the L\'{e}vy copula. L\'{e}vy copulas  differ from  proper copulas in some aspects: L\'{e}vy copulas are defined on $[-\infty, +\infty]^{d}$, can be unbounded  and depict dependence in a dynamic context while proper copulas  are defined on $[0,1]^{d}$, are bounded and correspond to a cross-sectional dependence. Proper copulas are functions that associate multidimensional cumulative distribution functions (cdfs) with their marginal cdfs: for  any cdf $F$ of a random vector $(X_{1}, \ldots,X_{d})$ in $\mathbb{R}^{d}$ with marginal cdfs $F_1, F_2, \ldots, F_d$, there exists a copula $C$ defined on $\overline{Ran(F_{1}) \times \cdots \times Ran(F_{d})}$, the closure of the image space of $(F_1, F_2, \ldots, F_d)$ such that
 \[
F(x_{1}, \ldots, x_{d})=C(F_1(x_{1}), \ldots, F_d(x_{d})), (x_{1}, \ldots, x_{d}) \in \mathbb{R}^{d}.
\]  
Moreover, the definition of the copula $C$ could be extended to  $[0,1]^{d}$ by extending the law of $(F_1(X_1), F_2(X_2), \ldots, F_d(X_d))$ to $[0,1]^{d}$. Examples of such extensions could be found in \cite{genest2007primer}. The definition of L\'{e}vy copulas parallels the definition of proper copulas. L\'{e}vy copulas are defined as functions linking the tail integral functions to the corresponding marginal tail integral functions. A generalization of the Sklar theorem to survival functions \citep{mcneil2015quantitative,mai2010extendibility} helps to establish their existence. In addition, L\'{e}vy copulas can be deduced from  proper copulas associated with the jumps $\Delta X_{t}$ for small $t$ \citep{kallsen2006characterization}. These similitudes elicit a natural question: is it possible to define a correspondence between L\'{e}vy copulas and proper copulas?

Using the lattice theory and the Scott topology, \citet[Theorem 1.33]{molchanov2005theory} constructed a correspondence between  L\'{e}vy measures concentrated on $[-\infty,+\infty]^{d}\setminus(+\infty, \ldots, +\infty)$ and  probability measures. Those results does not apply if the L\'{e}vy process has infinite activity in the neighbourhood of the origin.  In this  case, sets bounded away from the origin need to be considered in the definition of the correspondence but they are not easy to deal with in the lattice theory framework.The \citet{goldie1989records}'s construction of i.i.d. probabilistic structures  associated with a general partial order provides an example of how to solve the problem. Specifically, we consider in this article a partial order compatible with the inclusion of sets bounded away from the origin. We have the following results:
\begin{itemize}
\item [(a)] the L\'{e}vy measure can be represented as the hazard measure of a probability law;
\item [(b)] L\'{e}vy copulas sitting on $[0,+\infty]^{d}$ are in a one-to-one correspondence with proper copulas of max-id distributions.
\end{itemize}
In order to unify the presentation, we introduce the notion of exponential envelope families. The exponential envelope families of L\'{e}vy processes have been developed by  \citet{lauritzen1975general, kuchler1989exponential3}, and  \citet{kuchler2006exponential}. If $Cu(\chi)$ is the Laplace exponent (or the cumulant) associated with $\chi$ a non-negative semi-character on $(\mathbb{R}^{d},*)$ (i.e. $\chi$ is a real valued function satisfying $\chi(x * y)=\chi(x)\chi(y)$ \citep{clifford1961algebraic,lauritzen1975general} ), the exponential envelope family consists of a natural family of exponential distributions with density function
$$
h_{t}(x)=\frac{\chi(x)}{\exp(t Cu(\chi))}=\chi(x) \exp{\{-t Cu(\chi)\}}, 
$$
that may be derived from the expression of the Laplace transform
$$
\exp(t Cu(\chi))=E\{\chi(\wedge_{s \leq t} \Delta X_{s})\} < \infty, t>0
$$
with
$$
\wedge_{s \in \{s_{1},s_{2}\}}  \Delta X_{s}= \Delta X_{s_{1}}*\Delta X_{s_{2}}, \quad s_{1}, s_{2} \in [0,t].
$$
When the semi-character $\chi$ is the exponential function $x \mapsto \exp \langle \theta,x \rangle$ and the operation $*$ is the addition, the change of process is called 
Exponential Tilting, Esscher Tilting or Exponential Change of Measure \citep{kallsen2002cumulant,gerber1993option}.  When the semi-character $\chi$ is an indicator variable, $1_{A}$:
\begin{equation*}
	1_{A}(\omega) = \begin{cases}
		1 \text{ if } \omega \in A, \\
		0 \text{ if } \omega \in A^c,
	\end{cases}	
\end{equation*}
the set $A$ is called a filter. Examples of sets $A$ and  operations $*$ are:
\begin{itemize}
\item [(a)]$A$ is of the form $[l_1, +\infty] \times \cdots \times [l_d, +\infty]$ and the operation $*$ is the componentwise minimum; we have for $x=(x_1, \ldots, x_{d})$ and $y=(y_{1}, \ldots, y_{d})$,
$$
1_{A}(\min(x,y))=1_{A}(x) \times 1_{A}(y);
$$
\item [(b)]  $A$ is of the form $[-\infty,l_1] \times \cdots \times [-\infty,l_d]$ and the operation $*$ is the componentwise maximum; we have 
$$
1_{A}(\max(x,y))=1_{A}(x) \times 1_{A}(y).
$$
\end{itemize}
The exponential family  resulting from the transformation of the L\'{e}vy process $X$ is characterized by the following properties:
\begin{itemize}
\item [(a)]  the last observation is the sufficient statistic of the model;
\item [(b)]  the exponential family is an extreme point model in the families  that have the last observation as sufficient statistic \citep{kuchler1989exponential3}.
\end{itemize}
In other words, we are examining infinitely divisible distributions under a change of a measure that has interesting properties. The plan of the article is as follows.

 Section 2 presents the preliminary results on the general exponential model and the L\'{e}vy copula. The Laplace transform is widely used for its spectral properties but can also be analyzed from the point of view of an exponential change of measure \citep{kallsen2002cumulant,gerber1993option,
kuchler1989exponential3,kuchler1981analytical, kuchler2006exponential}. In the literature, infinitely divisible distributions  \citep{goldie1967class, steutel1967note, bondesson1981classes, barndorff2006some} have been studied from the point of view of the L\'{e}vy measure. It could be interesting to  characterize directly  infinitely divisible distributions by the properties of a probability measure instead of the L\'{e}vy measure. The dependence of multivariate infinitely divisible distributions can be closely associated with survival copulas \citep{mcneil2015quantitative, mai2010extendibility}. The Sklar's theorem for survival copula extended to tail integral functions help to establish the existence of L\'{e}vy copulas \citep{tankov2003financial,  kallsen2006characterization}. In order to propose a probabilistic framework for L\'{e}vy copulas, following \citet{barndorff2007levy}, we propose a non-negative $d-$increasing tail integral function. The corresponding  L\'{e}vy copula takes also non-negative values and coincides with the traditional L\'{e}vy copula on $[0,+\infty]^d$ \citep{barndorff2007levy}. We introduce the Archimedean copulas \citep{bauerle2008}, one of the most important families of proper copulas, and Archimedean L\'{e}vy copulas, as they will serve later for the illustration of the correspondence between L\'{e}vy copulas and proper copulas. The two parametric families are characterized by $d-$completely monotone functions \citep{mcneil2009multivariate,williamson1955multiply} that, we will show later, are related. Section 3 considers envelope exponential models where the semi-characters are filters. The operation $*$ is a maximum or a minimum with respect to a partial order that is related to the inclusion order of  sets bounded away from the origin.  The mean measure of the L\'{e}vy process is represented as a hazard measure associated with the distribution of the upper or lower record.  One of the contribution of this section is to establish the distributions of the upper and lower record under a partial order (other than the product order) which are not present in the literature. For comparison, the distribution of the lower record for the product order can be found in \citet[chap.5]{resnick2013extreme}. Section 4 presents a one-to-one correspondence between L\'{e}vy copulas and proper copulas: it make use of the distribution of the upper record and the proper copula associated. Section 5 illustrates how the correspondence works with the Archimedean copulas. It shows in particular that while generators of proper Archimedean copulas are Williamson $d-$transform of $F_{R}$, the cdf of the radial component of the associated random vector, generators of  Archimedean L\'{e}vy copulas are Williamson $d-$transform of $-\log(1-F_{R})$.
\section{Preliminaries}
\subsection{Notations and definitions}
Let us consider $H$  a nondecreasing function on $\mathbb{R}$; the generalized inverse of $H$ is the function $y \mapsto G^{-1}(y)$ with $G^{-1}$ defined as follows:
$$
G^{-1}(y)= \inf \{s: G(s) \geq y\}, y \in \mathbb{R}.
$$
Generalized inverses intervene in the inverse probability integral transform and in the sampling of random variables: if $U$ is uniformely distributed on $[0,1]$, and $G$ is a cumulative distribution function (cdf) then $G^{-1}(U)$ is a random variable with cdf $G$. Conversely, if $G$ is a continuous cdf of a random variable $V$, then $G(V)$ has a uniform distribution on $[0,1]$. In the theory of copulas, this transformation serves to transform the marginal cdfs of a multivariate cdf into uniform ones and it is called the probability integral transform. 

To state the preliminary results, we must introduce some notation. We will be working on $\overline{\mathbb{R}}^{d}$,  equipped with an operation $*$ such that $(\overline{\mathbb{R}}^{d}, *)$ is a semigroup, $d \geq 1$; unless precisely defined, the  relations and operations are to be taken componentwise. For example, the maximum over a set of vectors is the componentwise maximum:

$$
\max_{k \leq n} (X_{1k}, \ldots, X_{dk})=(\max_{k \leq n}X_{1k}, \ldots, \max_{k \leq n} X_{dk}).
$$
Inequalities are to be taken componentwise: if $x=(x_1, \ldots, x_{d})$, $y=(y_1, \ldots, y_{d})$,
$$
x < y \text{ means } x_{i}<y_{i}, \, i=1, \ldots,d, \text{ and } \\
x \leq  y \text{ means } x_{i} \leq y_{i}, \, i=1, \ldots,d.
$$
Rectangles will be denoted by
\begin{align*}
(a,b]&= \{x=(x_1, \ldots, x_{d}):a_{i} <x_{i} \leq b_{i}, i=1, \ldots, d\}, \text{ or }\\
[a,b)&=\{x=(x_1, \ldots, x_{d}):a_{i} \leq x_{i} < b_{i}, i=1, \ldots, d\},  \text{ or }\\
[a,b]&= \{x=(x_1, \ldots, x_{d}):a_{i} \leq x_{i} \leq b_{i}, i=1, \ldots, d\}.
\end{align*}
The tail integral functions are traditionally defined for sets $I(x_1) \times \cdots \times I(x_d)$ where $I(x_i)=[x i , + \infty)$ if $x_i \geq 0$ and $I(x_i)=(-\infty, x_i)$ if $x_i<0$, $i=1, \ldots,d$. We will instead consider the closed sets $J(x_1) \times J(x_2) \times \cdots \times J(x_d)$ with
$$
J(x_i)=[x i , + \infty] \text { if } x_i\geq 0 \text{ and } [-\infty, x_i]  \text { if }  x_i< 0, i \in \{1, \ldots,d\}.
$$
The closures are used essentially to ensure that the class of sets $J(x_1) \times J(x_2) \times \cdots \times J(x_d)$, $x=(x_1, \ldots, x_{d}) \in \overline{\mathbb{R}}^{d}$ is stable by countable union. If $L$ is the set  $J(l_1) \times J(l_2) \times \cdots \times J(l_d)$, then we will call $L^{c}$  the set $\overline{\mathbb{R}}^{\epsilon_1}\times \cdots \times \overline{\mathbb{R}}^{\epsilon_d}\setminus L$ where  $\epsilon_i=+$ if $l_i \geq 0$, $\epsilon_i=-$ if $l_i< 0$. For the definition of the probabilistic L\'{e}vy copula, to 
make sure that the tail integral $U^{+}$ is $d-$increasing, we will use $J^{+}$ instead of $J$ with
$$
J^{+}(x_i)=[1/x i , + \infty] \text { if } x_i\geq 0 \text{ and } [-\infty, x_i]  \text { if }  x_i< 0, i \in \{1, \ldots,d\}.
$$
For the clarity of exposition, we recall the important notions about L\'{e}vy copulas and the family of Archimedean L\'{e}vy copulas. In the same time, we propose a version of the L\'{e}vy copula living on ${\overline{\mathbb{R}}^{+}}^{d}$.
\subsection{ L\'{e}vy copulas and Archimedean L\'{e}vy copulas}
Let us consider a L\'{e}vy process  $X=\{X_{t} \in \mathbb{R}^{d}: t \geq  0\}$ with the triplet $(\beta,\gamma, \nu)$. It is made of a Gaussian component with parameter $(\beta,\gamma, 0)$ and a pure jump component $(0,0,\nu)$. Compound Poisson processes correspond to L\'{e}vy processes with a finite mean measure $\nu$. They are processes with a finite number of jumps in a given time span $[0,t]$; the copula corresponding to the c.d.f. of $X(t)$ can be expressed  in terms of the cdfs of the jumps. However, when the L\'{e}vy process has an infinite activity in any finite time, the L\'{e}vy measure is no more bounded and there does not exist a closed form expression of the  dependence of the process in term of the dependence of the jumps. \cite{tankov2003financial} and \cite{kallsen2006characterization}  proposed a way to characterize the dependence of a multivariate L\'{e}vy process that is applicable when the L\'{e}vy measure is not bounded. The method relies on applying the Sklar's theorem to the tail integral transform of the L\'{e}vy measure.
\begin{definition}
	Let $X=(X_1, \ldots, X_d)$ be a L\'{e}vy process in $\mathbb{R}^{d}$ with L\'{e}vy measure  $\nu$, the tail integral function corresponding to $\nu$ is a function $U: (\mathbb{R}\setminus\{0\})^{d} \rightarrow \mathbb{R}$ such that
	$$
	U(x_1, \ldots, x_d )=(\prod_{i=1}^{d}sgn(x_i))  \nu(I(x_1) \times I(x_2) \times \cdots \times I(x_d))
	$$
	where $I(x_i)=[x i , + \infty)$, $sgn(x_i)=1$ if $x_i \geq 0$, $I(x_i)=(-\infty, x_i)$, $sgn(x_i)=-1$ if $x_i<0$, $i=1, \ldots,d$;   $U$ is equal to zero if at least one of its argument is equal to $\infty$ , and $U$ is equal to $+\infty$  if all the 	arguments are equal to zero. 
\end{definition}
The tail integral function $U$ can take negative values. Here, we propose  versions of the tail integral function that are non-negative.  Let 
 $J^{+}(x_i)$ be the set $[1/x i , + \infty]$ for $x_i \geq 0$ and $J^{+}(x_i)=[-\infty, x_i]$ for $x_i<0$, $i=1, \ldots,d$; $J^{+}(x_1, \ldots,x_d)$ designates the product $J^{+}(x_1) \times J^{+}(x_2) \times \cdots \times J^{+}(x_d)$. Let us introduce absolutely monotone functions that will play an important role in the rest of this section.
 \begin{definition}	
A function $\psi$ defined on $\mathbb{R}$ is $d$-absolutely
	monotone if $\Delta_{h}^{k} \psi(t)=\sum_{i=0}^{k}(-1)^{k-i} \binom{k}{i}\psi(t+i h) \geq 0$  for $k \leq d$ and all $t$.  The function $\psi$ is $d$-completely monotone (respectively $d$-completely alternating) if, for $k \leq d$, $(-1)^{k}\Delta_{h}^{k} \psi \geq 0$ (respectively $(-1)^{k+1}\Delta_{h}^{k} \psi \geq 0$).
\end{definition}
\begin{definition}
	Let us consider the  integral functions $U_{l}^{+}: (\overline{\mathbb{R}})^{d} \rightarrow [0,\infty]$ and $U_{u}^{+}: (\overline{\mathbb{R}})^{d} \rightarrow [0,\infty]$  such that
	
\begin{align*}
U_{l}^{+}(x_1, \ldots, x_d )&= -\log(1- \exp(-\nu(\overline{\mathbb{R}}^{sign(x_1)}\times \cdots \times \overline{\mathbb{R}}^{sign(x_d)} \setminus J^{+}(x_1, \ldots,x_d)))	\\
U_{u}^{+}(x_1, \ldots, x_d )&= \nu(J^{+}(x_1, \ldots,x_d))).	
\end{align*}
	where $sign(x_i)=+$ if $x_i \geq 0$ and  $sign(x_i)=-$ if $x_i< 0$, $i=1, \ldots,d$.
\end{definition}

\begin{lemma}
\label{pre1}
Let us suppose that  $\nu(J^{+}(x_1) \times J^{+}(x_2) \times \cdots \times J^{+}(x_d))=0$ whenever  there exists $i \in \{1,\ldots,d\}$  such that $J^{+}(x_i) \subset \{+\infty,-\infty\}$  and $\nu([0, +\infty] \times \cdots \times  [0, +\infty] )=+\infty$. The functions $U_{l}^{+}$ and $U_{u}^{+}$ are $d-$increasing and grounded on any set of the form $\overline{\mathbb{R}}^{\epsilon_1}\times \cdots \overline{\mathbb{R}}^{\epsilon_d}$ with $\epsilon_i \in \{-,+\}$ and $i \in \{1,\ldots, d\}$.
\end{lemma}
In section 4, it will be shown that the applications $(x_{1}, \ldots, x_{d})\mapsto \exp(-\nu(\overline{\mathbb{R}}^{sign(x_1)}\times \cdots \overline{\mathbb{R}}^{sign(x_d)} \setminus J^{+}(x_1, \ldots,x_d))$  and $(x_{1}, \ldots, x_{d})\mapsto 1-\exp(-\nu(J^{+}(x_1, \ldots,x_d))$ are cdfs associated with some record processes. They are also $\max$-infinitely divisible distributions (in a sense to be defined) associated with some partial orders to be defined in section 3.
\begin{proof}
We will give a partial proof by relying on the assumption that there exists cdfs $G_{l,\epsilon_1 \cdots \epsilon_d }$ , $G_{u,\epsilon_1 \cdots \epsilon_d }$  on $\overline{\mathbb{R}}^{\epsilon_1}\times \cdots \overline{\mathbb{R}}^{\epsilon_d}$ such that 
\begin{align*}
U_{l}^{+}(x_1, \ldots, x_d )&= -\log(1- G_{l,\epsilon_1 \cdots \epsilon_d }(x_1, \ldots, x_d)),	\\
U_{u}^{+}(x_1, \ldots, x_d )&= -\log(1- G_{u,\epsilon_1 \cdots \epsilon_d }(x_1, \ldots, x_d)),	
\end{align*}
$(x_1, \ldots, x_d ) \in \overline{\mathbb{R}}^{\epsilon_1}\times \cdots \overline{\mathbb{R}}^{\epsilon_d}$. Since the function $s \mapsto -\log(1-s)$ is absolutely monotone,  and the functions $G_{l,\epsilon_1 \cdots \epsilon_d }$,$G_{u,\epsilon_1 \cdots \epsilon_d }$ are $d-$increasing, then $U_{l}^{+}$ and $U_{u}^{+}$  are also $d-$increasing \citep{morillas2005characterization}. Since the cdfs are grounded, the functions $U_{l}^{+}$ and $U_{u}^{+}$ are also grounded.
\end{proof}
For a continuous survival function $G$ defined on $\mathbb{R}^{d}$,  an extension of the Sklar's theorem to survival functions (see \cite{mcneil2015quantitative} or \cite{mai2010extendibility}) states that there exists a unique $d$-dimensional copula linking the distribution $G$ and the marginal survival functions. The same idea applies to the tail integral of a L\'{e}vy measure: if the L\'{e}vy measure is non-atomic, there exists a grounded $d-$increasing function (the L\'{e}vy copula) linking the tail integral of the L\'{e}vy measure and the marginal tail integrals.
\begin {definition}[L\'{e}vy copula]
The L\'{e}vy copula is a function $F: (-\infty,+\infty]^{d}\rightarrow  [-\infty, +\infty]$ such that
\begin{itemize}
	\item [(a)] $F(x_1, \ldots,x_d)=+\infty$ if $x_1 =\ldots= x_d = +\infty$;
	\item [(b)] $F(x_1, \ldots,x_d)=0$ if for at least one $i \in \{1,2, \ldots,d\}$, $x_i=0$;
	\item [(c)] $F$ is $d-$increasing ( the volume  of a set $(a_1, b_1]\times \cdots \times (a_d, b_d]$ is positive);
	\item [(d)] the one-dimensional marginal $F_i$ satisfies $F_i(x_i)=x_i$, $x_i \in  (-\infty,+\infty]$, $i \in \{1, \ldots,d\}$.
\end{itemize}
\end{definition}
The L\'{e}vy copula $F$ corresponds to the tail integral function $U$ and it can take negative values. For the tail integral functions $U_{l}^{+}$ or $U_{u}^{+}$, let us introduce a corresponding L\'{e}vy copula  called the probabilistic L\'{e}vy copula $F^{+}$. It is a L\'{e}vy copula living on $[0,+\infty]^{d}$. For the rest of this section,  $U_{l}^{+}$ or $U_{u}^{+}$ will be called indistinctly $U^{+}$.
\begin {definition}[Probabilistic L\'{e}vy copula]
The probabilistic L\'{e}vy copula $F^{+}$ is a L\'{e}vy copula defined on $[0,+\infty]^{d}$ with values in $[0, +\infty]$. In particular, it verifies the properties of a L\'{e}vy copula restricted to  $[0,+\infty]^{d}$:
\begin{itemize}
	\item [(a)] $F^{+}(x_1, \ldots,x_d)=+\infty$ if $x_1 =\ldots= x_d = +\infty$;
	\item [(b)] $F^{+}(x_1, \ldots,x_d)=0$ if for at least one $i \in \{1,2, \ldots,d\}$, $x_i=0$;
	\item [(c)] $F^{+}$ is $d-$increasing ( the volume  of a set $(a_1, b_1]\times \cdots \times (a_d, b_d]$ is positive);
	\item [(d)] the one-dimensional marginal $F^{+}_i$ satisfies $F^{+}_i(x_i)=x_i$, $x_i \in  [0,+\infty]$, $i \in \{1, \ldots,d\}$.
\end{itemize}
\end{definition}
Let $\overline{\prod_{i=1}^{d}\ R_i}$ be the closure of the range of the function $(U_1, U_2, \ldots, U_{d})$ (respectively $(U_1^{+}, U_2^{+}, \ldots, U_{d}^{+})$) where $U_1, \ldots, U_d$ (respectively $U_1^{+}, U_2^{+}, \ldots, U_{d}^{+}$) are the marginal tail integral functions.
\begin{lemma}
	Let $X$ be a multidimensional  L\'{e}vy process in $\mathbb{R}^{d}$. For a non-atomic  L\'{e}vy measure, we can associate uniquely a L\'{e}vy copula (respectively a probabilistic L\'{e}vy copula) defined by the relationship:
\begin{equation}
\label{eqlevy1}
U(x_1, \ldots,x_d)=F(U_1(x_1), \ldots, U_d(x_d))
\end{equation}
(respectively
\begin{equation}
\label{eqlevy2}
U^{+}(x_1, \ldots,x_d)=F^{+}(U_1^{+}(x_1), \ldots, U_d^{+}(x_d))
\end{equation}
).\\

If the  L\'{e}vy measure  is atomic,  the L\'{e}vy copula (respectively the probabilistic L\'{e}vy copula ) is uniquely defined only on $\overline{\prod_{i=1}^{d}\ R_i}$ (respectively $\overline{\prod_{i=1}^{d}\ R_i^{+}}$) where $R_i$ (resp. $R_i^{+}$ ) is the range of $U_i$ (resp. $U_i^{+}$ ).
\end{lemma}
 \begin{proof}
For the equation \ref{eqlevy1}, see  \citet[Theorem 3.6]{kallsen2006characterization} for a proof. For the equation \ref{eqlevy2}, we can use the Sklar's theorem applied to the cdfs $G_{l,\epsilon_1 \cdots \epsilon_d }$ or $G_{u,\epsilon_1 \cdots \epsilon_d }$ and consider the distortion of the resulting copula by the absolutely monotone function $s \mapsto -\log(1-s)$: the result is a $d$-increasing and grounded function that satisfies the definition of a probabilistic L\'{e}vy copula (see  \citet[Theorem 3.6]{morillas2005characterization} for an example of a distortion of a proper copula by an absolutely monotone function).
 \end{proof} 
%
\begin{definition}
	Let us consider $\psi$ a  continuous strictly decreasing function defined on $[0, +\infty]$ with values in $[0,1]$ such that $\psi(0)=1$, $\psi(\infty)=0$. Let  $\psi^{-1}$ denote the pseudo-inverse of $\psi$. A proper copula $C$ is called an Archimedean copula if it admits the functional form:
	\[
	C(u_{1}, \ldots, u_{d})=\psi(\psi^{-1}(u_1)+\cdots+\psi^{-1}(u_d)), (u_{1}, \ldots, u_{d})\in [0,1]^{d}.
	\]
\end{definition}
 \cite{mcneil2009multivariate}  showed that the set of generators $\psi$ of Archimedean copulas are the $d-$completely monotone functions $\mathcal{\Psi}_{d}$:
\begin{align*}
	\mathcal{\Psi}_{d}=\{\psi: [0, \infty] \rightarrow [0,1]: \psi(0)=1, \psi(\infty)=0,& (-1)^{j}\psi^{(j)}\geq 0,  j \leq d-2, \text{ and } \\
	&(-1)^{d-2}\psi^{(d-2)} \text{ decreasing and convex.} \}
\end{align*}
Any $d-$completely monotone function $\psi$ has also the following integral representation \citep{mcneil2009multivariate, williamson1955multiply}:
$$
\psi(x)=\int \max(0, 1- at)^{d-1} d\gamma(a), d \geq 2,
$$
where $\gamma$ represents the cdf of a positive random variable. Inversely, $\gamma$ has the following representation \citep[Theorem 2]{williamson1955multiply}:
$$
\gamma(u)=\frac{(-1)^{d-2}}{d-2} \int_{0^{+}}^{u}\frac{1}{x^{d-2}}d\psi^{(d-2)}(\frac{1}{x}), u>0.
$$
Archimedean copulas can be equivalently expressed in a multiplicative form by considering $\rho$  such that $\psi(\cdot)=\rho(\exp(-\cdot))$:
\[
C_{\rho}(u_1, \ldots, u_{d})=\rho(\prod_{i=1}^{d}\rho^{-1}(u_{i})), (u_{1}, \ldots, u_{d}) \in [0,1]^{d};
\]
$\rho$ is a multiplicative generator. It appears that  Archimedean copulas are the result of a transformation of the independence copula $C(u_1, \ldots, u_{d})=\prod_{i=1}^{d}u_{i}$ by the generator $\rho$. 
For L\'{e}vy copulas, the definition of the Archimedean copulas remains the same, except that the range of the generator changes.
\begin{definition}[\cite{bauerle2008}]
	Let us consider $\overline{\psi}: (0,+\infty)\rightarrow (0,+\infty) $ a $d-$completely monotone function, with $\lim_{x \rightarrow 0} \overline{\psi}(x)=\infty$ and $\lim_{x \rightarrow \infty} \overline{\psi}(x)=0$, and the function
	\[
	F(x_{1}, \ldots, x_{d})=\overline{\psi}(\overline{\psi}^{-1}(x_1)+\cdots+\overline{\psi}^{-1}(x_d)),
	\]
	$(x_{1}, \ldots, x_{d}) \in (0,+\infty)^{d}$. The function $F$ is  an Archimedean L\'{e}vy copula.
\end{definition}

\subsection{General exponential family}
Let us consider the commutative semi-group  $(\overline{\mathbb{R}}^{d},*)$;
\begin{itemize}
	\item [(i)] $\forall x, y, z \in \overline{\mathbb{R}}^{d}, x*(y*z)=(x*y)*z$,
	\item [(ii)] $\forall x, y \in \overline{\mathbb{R}}^{d}, x*y=y*x$.    
\end{itemize}	
\begin{definition}
	The space of  homomorphisms $\chi$ defined on $(\overline{\mathbb{R}}^{d},*)$ with values in $\mathbb{R}^{+}$, i.e. satisfying for all $x,y \in \mathbb{R}^{d}$
	\[
	\chi(x*y)=\chi(x) \cdot \chi(y), x, y \in \overline{\mathbb{R}}^{d}
	\]
	is called $L(\overline{\mathbb{R}}^{d}, \mathbb{R}^{+})$; $L(\overline{\mathbb{R}}^{d}, \mathbb{R}^{+})$ is equipped with the operation $\otimes$ defined as follows:
	\[
	\forall \chi_{1}, \chi_{2} \in L(\overline{\mathbb{R}}^{d}, \mathbb{R}^{+}), x \in \overline{\mathbb{R}}^{d},  \chi_{1}\otimes\chi_{2}(x)= \chi_{1}(x)\chi_{2}(x);
	\] 
	$(L(\overline{\mathbb{R}}^{d}, \mathbb{R}^{+}), \otimes)$ forms a semi-group and its elements are called semi-characters.
\end{definition}
\begin{definition}
	A set $A \subset \overline{\mathbb{R}}^{d}$ is called a filter if 
	\begin{itemize}
		\item [(i)] $(A,*)$ is a subsemigroup of $(\overline{\mathbb{R}}^{d},*)$ and 
		\item [(ii)] for $x, y \in \overline{\mathbb{R}}^{d}$ such that $x*y \in A$, we have $x \in A$ and  $y \in A$.
	\end{itemize}
\end{definition}	
For any element $\chi$ of $L(\overline{\mathbb{R}}^{d}, \mathbb{R}^{+})$, we can define a filter $A$ in the following way:
\[
A=\{x: \chi(x)>0\};
\]
inversely any filter $A$ can be associated with the semi-character $\chi=1_{A}$. Semi-characters serve to generalize the notion of Laplace transform to semigroups.
 If $Cu(\chi)$ is the Laplace exponent associated with the semi-character $\chi$ and the probability measure $\mu_{t}$, we have 
 $$
 \exp(t Cu(\chi))=E\{\chi(\wedge_{s \leq t} \Delta X_{s})\} < \infty, t>0.
 $$
 where $\wedge_{s \in \{s_1, s_2\}} \Delta X_{s})=\Delta X_{s_1}*\Delta X_{s_2}$.  The Laplace transform, $\exp(t Cu(\chi))$, when it is finite, defines a  family of probability measures $\{h_{t}\mu_{t}: t>0\}$ with density function
$$
h_{t}(x)=\frac{\chi(x)}{\exp(t Cu(\chi))}=\chi(x) \exp{\{-t Cu(\chi)\}}, t>0.
$$
When the function $\chi$ is an exponential function and the operation $*$ is the addition, the transformation $X_{t} \mapsto h_{t}(\sum_{s \leq t} \Delta X_{s})= h_{t}(X_{t})$ is called an Exponential Tilting, an Esscher Tilting or an Exponential Change of Measure \citep{kallsen2002cumulant, gerber1993option}.  Moreover, the process $\{h_{t}(X_{t})=\chi(X_{t}) \exp{\{-t Cu(\chi)\}}: t>0\}$ is a martingale for every semi-character $\chi$.

Let us consider a filtered space $(\Omega, \mathcal{F},  \{\mathcal{F}_{t}: t>0\})$ with $\Omega$ the space  of the compositions  $\{\wedge_{s\leq t} \Delta X_{s}: t>0 \}$ of the  jumps of a pure jump L\'{e}vy process and $ \mathcal{F}_{t}$ the filtration generated by the compositions.

\begin{definition}
 Let us call $\mu_{t}$ the distribution of $\wedge_{s\leq t} \Delta X_{s}$, $t>0$. A class $\mathcal{P}=\{\mu^{\chi}: \chi \text{ semi-character}\}$ of probability measures on the filtered space $(\Omega, \mathcal{F},  \{\mathcal{F}_{t}: t>0\})$  
is called an exponential family if 
		$$
		\frac{d\mu_{t}^{\chi}}{d\mu_{t}}(x)=\chi(x) \exp{\{- Cu(\chi)t\}}=\exp{\{k(\chi)x - Cu(\chi)t\}}, 
		$$
		and
		$\mu_{t}^{\chi}$ is the restriction of $\mu^{\chi}$ to $\mathcal{F}_{t}$,  $x \in \mathbb{R}^{d}$, $t>0$. 
	\end{definition}
\citet{kuchler1981analytical} showed that, under the change of measure, the characteristic function $\rho$ has the following expression 
\begin{eqnarray*}
\rho(z)&=&\exp \left\{i \langle z,\beta^{'} \rangle - \langle \gamma^{'} z, z \rangle +\int\{ e^{i  \langle z, x  \rangle}-1 -i \frac{\langle z, x \rangle }{1+\langle x, x \rangle}\}1_{D}(x) \nu'(dx) \right\},
\end{eqnarray*}
with $D=\{x \in \mathbb{R}^{d}: \lVert x \rVert \leq 1\}$ and $\int\{ \frac{\langle x, x \rangle }{1+\langle x, x \rangle}\}1_{D}(x) \nu(dx) < \infty$. In particular, L\'{e}vy processes eligible to be part of the exponential family are those with jumps bounded above or below i.e the jumps are in a set $J(l_1)\times \cdots \times J(l_d)$, $(l_1, \ldots, l_d) \in \mathbb{R}^{d}$.
\begin{lemma}[Characteristics of the transformed process]
If $\nu'$ is the L\'{e}vy measure associated with the transformed process, we have 
\begin{align*}
\nu'(dx)&=\chi(x)\nu(dx), x \in \overline{\mathbb{R}}^{d}.
\end{align*}
\end{lemma}
\begin{proof}
	See \citet[Theorem 2]{kuchler1981analytical} for the one-dimensional result and \citet[chapter 11, section 5]{ kuchler2006exponential} for $d \geq 1$.
\end{proof}

\begin{lemma}
\label{lemMSuf}
For a given semi-character $\chi$, $\{t \mapsto \mu_{t}^{\chi}: t \geq 0\}$ forms a convolution semi-group. If $(L(\overline{\mathbb{R}}^{d}, \mathbb{R}^{+}), \otimes)$ separates points, then the application $(\Delta X_s, s \leq t) \mapsto (\wedge_{s \leq t}\Delta X_s)$ is a minimal sufficient statistic.
\end{lemma}
\begin{proof}
The family $t \mapsto \mu_{t}$ forms a convolution semi-group: 
$$
\mu_{t+s}=\mu_{t} \mu_{s},  t,s \geq 0.
$$
In addition, for $x, y \in L$, the density $h_{t}$ verifies the rule $h_{t+s}(x*y)=h_{t}(x)h_{s}(y)$ which makes the family $\{\mu_{t}^{\chi}: t \ge 0\}$  a convolution semi-group. By definition of the exponential model,  the application $(\Delta X_s, s \leq t) \mapsto \wedge_{s \leq t} \Delta X_s$ is a sufficient statistic. If  $(L(\overline{\mathbb{R}}^{d}, \mathbb{R}^{+}), \otimes)$ separates points, the Laplace transform is injective and all the information about $\chi$  can be recovered by knowing $\exp(tf(\chi))$. In this case, according to the \cite{lehman1950}'s criterion,  the application $(\Delta X_s, s \leq t) \mapsto \wedge_{s \leq t}\Delta X_s$ is minimal sufficient: 
$$
\frac{h_t(x)}{h_t(y)}=\frac{\chi(x)}{\chi(y)}=c,\,\forall{\chi} \iff x=y \text{ if } (L(\overline{\mathbb{R}}^{d}, \mathbb{R}^{+}), \otimes) \text{ separates points.}
$$
\end{proof}
\citet{kuchler1989exponential3} showed that the exponential family is extremal among the models containing the initial L\'{e}vy process and having the last observation as the sufficient statistic. Even more, they constitute the minimal part of the Martin boundary of the initial L\'{e}vy process. As a result,  some convex classes of infinitely divisible distributions are mixtures of exponential families. An example of such mixtures is the mixture  $\int_{[0,t]} f(s)dX_{s}$ where 
$f(s)$, $s \in [0,t]$, is a random variable on $[0,\infty)$ with distribution $\tau_{s}$ and $X$ is a L\'{e}vy process with a cumulant function $Cu$ satisfying $\int |Cu(f(s)y)|ds <\infty$ for almost all $y$. \citet{goldie1967class} studied the degenerate case $f(s)=Y 1_{s_{0}}(s)$,	 $X_{s_0}$ exponentially-distributed, \citet{steutel1967note} studied discrete mixtures with $f(s)$ deterministic and $X_{s}$ exponentially-distributed,$s>0$. \citet{bondesson1981classes, barndorff2006some} studied the general case of a deterministic $f$ respectively in dimension $d=1$ and $d>1$.
\section{Exponential families generated by filters}
\subsection{The law of the upper and lower record of the jumps}
In this section, the operation $*$ is a maximum or a minimum  defined on a set  partially ordered by the relation $\leq$:
$$
x*y=x	\iff  x \leq y,  \text{ or }  x*y=x  \iff y \leq x.  
$$
On the semigroup $(\overline{\mathbb{R}}^{d},*)$, we consider a semi-character of the form $1_{L}$:
\begin{equation*}
	1_{L}(x) = \begin{cases}
		1 \text{ if } x \in L, \\
		0 \text{ if } x \in L^c,
	\end{cases}	
\end{equation*}
where  $L \subset \overline{\mathbb{R}}^{d}$ is  a filter. The structure on the semigroup $(\overline{\mathbb{R}}^{d},*)$ depends strongly on the definition of the order relation $\leq$.

\begin{definition}
A partially ordered set (poset) $L \subset \mathbb{R}^{d}$ is a nonempty set  equipped with a transitive, reflexive and anti symmetric relation $ \leq$: for $x,y,z \in L$,
	\begin{itemize}
		\item [(i)] $x \leq y$ and $y \leq x$ imply $x=y$;
		\item [(ii)] $x \leq x$;
		\item [(iii)] $x \leq y$ and $y \leq z$ imply $x \leq z$;
	\end{itemize}
	
A lattice is a poset in which every nonempty finite subset has a greatest lower bound (infimum) and a least upper bound(supremum). The lattice $L$ is called a complete lattice if the infimum and supremum exist  not only for nonempty finite subsets but also for all nonempty subsets. 
\end{definition}

\begin{example}[Product order on $\overline{\mathbb{R}}^{d}$]
The product order $\leq $ on  $\overline{\mathbb{R}}^{d}$ is the componentwise order such that $(x_{1}, \ldots, x_{d}) \leq (y_{1}, \ldots, y_{d})$ if and only if $x_{1} \leq y_{1}$, $\ldots$, $x_{d} \leq y_{d}$. The graph $G$ of the relation $\leq$ is defined as the set of pairs $(x,y)$ such that $x \leq y$,
\begin{align*}
G&=\{(x,y)\in \overline{\mathbb{R}}^{d} \times \overline{\mathbb{R}}^{d}: x=(x_1, \ldots,x_d), y=(y_1, \ldots, y_d), x_i \leq y_i, 1 \leq i \leq d\}\\
&=\{(x_1,y_1) \in \overline{\mathbb{R}} \times \overline{\mathbb{R}}: x_1 \leq y_1\}^{d}.
\end{align*}
The graph  $G$ is product-measurable (measurable for the product $\sigma-$algebra) and is closed. The $x-$sections of $G$, $\{y: x \leq y\}=[x, +\infty]$,  $x \in \overline{\mathbb{R}}^{d}$ are filters for the semigroup $(\overline{\mathbb{R}}^{d}, \max)$ where $\max$ is the component-wise maximum.
\end{example}

\begin{definition}
Assume that $L=J(l_1) \times \cdots \times J(l_d)$, $l=(l_1, \ldots, l_{d})\in \overline{\mathbb{R}}^{d}$ with
$$
J(l_i)=[l_i, +\infty] \text{ if  }  l_i \geq 0 \text{ and } J(l_i)=[-\infty, l_i] \text{ if  }  l_i<0, i=1, \ldots, d.
$$
The relation   $\leq $ (respectively $\geq$ )  is defined on $L$  as the componentwise ``lower or equal" (respectively ``greater or equal" )  for positive values and the componentwise ``greater or equal" (respectively ``lower or equal" )  for negative values:	
\begin{align*}
x, y \in L,  x \leq y \iff  x_i \leq y_i \text{ for } i \text{ such that } l_i \geq 0 \text{ and }  y_i \leq x_i \text{ for } i&  \text{ such that } l_i < 0, \\
&
i \in \{1, \ldots,d\};
\end{align*}
(respectively
$$
 x \geq y \iff x_i \geq y_i \text{ for } i \text{ such that } l_i \geq 0 \text{ and }  y_i \geq x_i \text{ for } i \text{ such that }  l_i < 0,  i \in \{1, \ldots,d\}
$$
).
\end{definition}
\begin{lemma}
$(L, \leq )$ and $(L, \geq )$ are lattices.
\end{lemma}
\begin{proof}
	The relation $\leq$ is reflexive: for any $x \in L$, $x \leq x$; 	the relation $\leq$ is antisymmetric (respectively transitive): if $x \leq y$ and $y \leq  x$ (respectively $y \leq z$) then $x=y$ (respectively $x\leq z$) since $x_i \leq y_i$ and $y_i \leq  x_i$ (respectively $y_i \leq  z_i$ ) implies that $x_i=y_i$ (respectively $x_i \leq  z_i$ )  for $i$ such that  $l_i \geq 0$ or $l_i<0$. Moreover, for any $x,y \in L$, $\min(x,y)\in L$ and $\max(x,y)\in L$. The proof for $(L, \geq )$ uses similar arguments.
\end{proof}
\begin{definition}
We recall that $L=J(l_1) \times \cdots \times J(l_d)$, $l=(l_1, \ldots, l_{d})\in \overline{\mathbb{R}}^{d}$ with
	$J(l_i)=[l_i, +\infty]$  if  $l_i \geq 0$ and $J(l_i)=[-\infty, l_i]$ if $l_i<0$, $i=1, \ldots,d$. Let us define on $L$ the operation $*$  and the corresponding relation $\leq $(respectively $\tilde{*}$ and the corresponding relation $\geq$ ) :	
$$
x, y \in L,  x*y=z  \iff z=(z_1, \ldots,z_d) \text{ and for } i \in \{1, \ldots,d\},
$$
\begin{equation*}
	z_{i} = \begin{cases}
		\min(x_{i}, y_{i}) \text{ if } l_i \geq 0, \\
		\max(x_{i}, y_{i}) \text{ if } l_i < 0,
	\end{cases}	
\end{equation*}
and
$$
 x \leq y \iff x*y=x;
$$
(respectively 
$$
x, y \in L, x\tilde{*}y=z  \iff z=(z_1, \ldots,z_d) \text{ and for } i \in \{1, \ldots,d\},
$$
\begin{equation*}
	z_{i} = \begin{cases}
		\max(x_{i}, y_{i}) \text{ if } l_i \geq 0, \\
		\min(x_{i}, y_{i}) \text{ if } l_i < 0,
	\end{cases}	
\end{equation*}
and
$$
x \geq y \iff x \tilde{*}y=y
$$
). \\
We will call  $\mu_{t}$ the distribution of $\wedge_{s \leq t} \Delta X_{s}$ and $\tilde{\mu}_{t}$ the distribution of $\vee_{s\leq t} \Delta X_{s}$, $t>0$ with
$$
\wedge_{s \in \{s_1,s_2\}} \Delta X_{s}=\Delta X_{s_1}*\Delta X_{s_2}
$$
and
$$
\vee_{s \in \{s_1,s_2\}} \Delta X_{s}=\Delta X_{s_1}\tilde{*}\Delta X_{s_2}.
$$
The statistics $\wedge_{s\leq t} \Delta X_{s}$ and  $\vee_{s\leq t} \Delta X_{s}$ are called respectively the lower and upper records. They are respectively the running minimum and the running maximum.
\end{definition}
\begin{lemma}
For any filter $L$, let us call $\mathring {L}$ the interior of $L$ and $L^{c}$ the set $\overline{\mathbb{R}}^{\epsilon_1}\times \cdots \times \overline{\mathbb{R}}^{\epsilon_d}\setminus L$ where $L=J(l_1)\times \cdots \times J(l_d)$, $\epsilon_i=+$ if $l_i \geq 0$, $\epsilon_i=-$ if $l_i< 0$; we have:
\begin{align}
\label{eqfilter}
1_{L}(\wedge_{s \leq t} \Delta X_s)&=\prod_{s \leq t} 1_{L}(\Delta X_s), \nonumber \\
	\\
1_{\mathring {L}^{c}}(\vee_{s \leq t} \Delta X_s)&=1_{\overline{\mathbb{R}}^{\epsilon_1}\times \cdots \overline{\mathbb{R}}^{\epsilon_d}\setminus\mathring {L}}(\vee_{s \leq t} \Delta X_s)=\prod_{s \leq t} 1_{\mathring {L}^{c}}(\Delta X_s).\nonumber
\end{align}
If  $L_1, L_2$ are filters, $\mathring{L}_{1}$,$\mathring{L}_{2}$,  $\mathring{L}_{1} \cap  \mathring{L}_{2}$ are also filters.
\end{lemma}
\begin{proof}
If  $L_i$ is a filter, it is of the form $Jl_{i1})\times \cdots \times J(l_{id})$ and the interior of $L_i$ is a filter, $i=1,2$. Since the intersection of two filters is a filter, if  $L_1, L_2$ are filters then $\mathring{L}_{1} \cap  \mathring{L}_{2}$ is also filter.

By definition of a filter, the proposition $1_{L}(\wedge_{s \leq t} \Delta X_s)=\prod_{s \leq t} 1_{L}(\Delta X_s)$ is true. It is equivalent to the following assertion:
$$
\wedge_{s \leq t} \Delta X_s \in L \iff \Delta X_s \in L \text { for } s \leq t.
$$	
Similarly,  the equivalence
$$
\vee_{s \leq t} \Delta X_s \notin \mathring {L} \iff \Delta X_s \notin \mathring {L} \text { for } s \leq t
$$
 is true because since $\mathring {L}$ is an upper set (for the relation $\geq$) and is open, there exists $s \leq t$ such that $\Delta X_s$ is in $\mathring {L}$ if and only if $\vee_{s \leq t} \Delta X_s$ is in $\mathring {L}$.
\end{proof}
Let us call $\mu_{t}$ the distribution of $\wedge_{s\leq t} \Delta X_{s}$ and  $\tilde{\mu}_{t}$ the distribution of $\vee_{s\leq t} \Delta X_{s}$.  Let us divide the time axis into time spans $[0,t]$, $(t,2t]$, $\ldots$, $((n-1)t,nt]$, $\ldots$,  and note $X_{n}(t)=\wedge_{(n-1)t <s \leq nt}\Delta X_{s}$ (resp. $\tilde{X}_{n}(t)=\vee_{(n-1)t <s \leq nt}\Delta X_{s}$ ), $n \geq 1$. Let us define multivariate records compatible with the filters in the following way:
\begin{itemize}
\item [(i)] $X_{n}(t)=(X_{n,1}(t), \ldots,X_{n,d}(t))$ is a record if, for $k<n$, $X_{k}(t) \notin J(X_{n,1}(t)) \times\cdots \times J(X_{n,d}(t))$ i.e $X_{n}(t)$ is a record if and only if $X_{k}(t) \ngeq X_{n}(t) $, $k <n$;
\item [(ii)] $\tilde{X}_{n}(t)=(\tilde{X}_{n,1}(t), \ldots,\tilde{X}_{n,d}(t))$ is a record if, for $k<n$, $\tilde{X}_{k}(t) \in \mathring{J}(\tilde{X}_{n,1}(t)) \times\cdots \times \mathring{J}(\tilde{X}_{n,d}(t))$ i.e $X_{n}(t)$ is a record if and only if $X_{k}(t) \geq X_{n}(t)$, $k <n$, the inequality being strict for at least one component.
\end{itemize}

 We will call $X_{n}(t)$ a lower record and $\tilde{X}_{n}(t)$ an upper record. Given the lower record process (respectively the upper record process), let us call $N_{t}$ (respectively $\tilde{N}_{t}$) the point process  giving the number of lower records (upper records) that fall in a measurable set. Let us call 
$H_{t}$ (resp. $\tilde{H}_{t}$) the hazard measure defined for any filter $L$ by
\begin{align*}
H_{t}(L)&=-\log(\mu_{t}^{c}(L))+\sum_{l \in D_{\mu_{t}}  \cap L}
\frac{\mu_{t}(l)}{\mu_t(J(l_1)\times \cdots \times J(l_d))}\\
&=\int_{L \cap S_{\mu_{t}}}\frac{1}{\mu_t(J(l_1)\times \cdots \times J(l_d))}\mu_{t}(dl),
\end{align*}
(resp.
\begin{align*}
\tilde{H}_{t}(\mathring{L})&=-\log(1-\tilde{\mu}_{t}^{c}(\mathring{L}))+\sum_{l \in D_{\tilde{\mu}_{t}} \cap \mathring{L} }\frac{\tilde{\mu}_{t}(l)}{1-\tilde{\mu}_{t}(\mathring{ J}(l_1)\times \cdots \times \mathring{J}(l_d))}\\
&=\int_{L \cap S_{\tilde{\mu}_{t}}}\frac{1}{1-\tilde{\mu}_{t}(\mathring{ J}(l_1)\times \cdots \times \mathring{J}(l_d))}\tilde{\mu}_{t}(dl),
\end{align*}
)
where $S_{\mu_{t}}$ (resp. $S_{\tilde{\mu}_{t}}$ ) is the support of $\mu_{t}$
(resp. $\tilde{\mu}_{t}$ ), $D_{\mu_{t}}$ (resp. $D_{\tilde{\mu}_{t}}$ ) is the set of atoms of $\mu_{t}$
(resp. $\tilde{\mu}_{t}$ ), $\mu_{t}^{c}$ (resp. $\tilde{\mu}_{t}^{c}$ ) is the continuous part of  $\mu_{t}$
(resp. $\tilde{\mu}_{t}$ ), $t>0$. We can extend the definition of  $H_{t}$ (resp. $\tilde{H}_{t}$) to  the sigma-algebra generated by the filters as follows: for any measurable set $A$, we have 
$$
H_{t}(A)=\int_{A \cap S_{\mu_{t}}}\frac{1}{\mu_t(J(l_1)\times \cdots \times J(l_d))}\mu_{t}(dl),
$$
(resp.
$$
\tilde{H}_{t}(A)=\int_{A \cap S_{\tilde{\mu}_{t}}}\frac{1}{1-\tilde{\mu}_{t}(\mathring{ J}(l_1)\times \cdots \times \mathring{J}(l_d))}\tilde{\mu}_{t}(dl)
$$
), $t>0$. The component-wise supremum of the jumps of a L\'{e}vy process (for the product order) is known to be max-infinitely divisible (max-id) \citep{falk200830}.  A cdf $G$ is max-id if there exist a measure $\nu$ such that $G(x)=\exp(-\nu(\mathbb{R}^{d}\setminus(-\infty,x]))$, $x \in \mathbb{R}^{d}$  \citep{balkema1977max}. The existence of the exponent measure is implied by the following characterization.
 \begin{definition}
 	A distribution function $G$ on $\mathbb{R}^{d}$ is max-infinitely divisible (max-id) if, for every integer $k>1$, there exists a cdf $G_{k}$  such that $G(x)=(G_{k}(x))^{k}$ for every $x \in\mathbb{R}^{d}$.
 	Similarly, a survival function $G$ on $\mathbb{R}^{d}$ is said to be  min-infinitely divisible (min-id) if for every $k>1$ there exists a survival function $G_{k}$ on $\mathbb{R}^{d}$ such that $G(x)=(G_{k}(x))^{k}$ for every $x \in\mathbb{R}^{d}$.
 \end{definition}
In the present setting, the supremum for the relation $\leq$ is not exactly the component-wise supremum: the supremum corresponds to the component-wise supremum only for positive values.  The lemma \ref{exponent} shows that $\nu$ is the exponent measure for the restriction of $\mu_t$ and $\tilde{\mu}_t$ on filters. Before stating this lemma, let us introduce the hitting functional and the avoidance functional.

\begin{definition}[Hitting functional and avoidance functional]
The probability that  none of the jumps  of the L\'{e}vy process with L\'{e}vy measure $\nu$, during the period $[0,t]$, falls in a measurable set $A$ is called the avoidance functional and is noted $Q(A)=\exp(-t\nu(A))$. 	The probability that the jumps of the L\'{e}vy process hit a set $A$ during the time $[0,t]$ is called the hitting functional and is noted $T(A)=1-\exp(-t\nu(A))$.
\end{definition}
\begin{lemma}
\label{exponent}
For any filter $L$, we recall that $L^{c}$ is the set $\overline{\mathbb{R}}^{\epsilon_1}\times \cdots \times \overline{\mathbb{R}}^{\epsilon_d}\setminus L$ where $L=J(l_1)\times \cdots \times J(l_d)$, $\epsilon_i=+$ if $l_i \geq 0$, $\epsilon_i=-$ if $l_i< 0$; we have	
\begin{align}
\label{eqexpo}
\mu_{t}(L^{c})&=T(L^{c})=1-\exp(-t\nu(L^{c})),\nonumber \\
 \\
\tilde{\mu}_{t}(\mathring{L})&=T(\mathring{L})=1-\exp(-t\nu(\mathring{L})). \nonumber	
\end{align}
 The family of exponential  distributions associated with  $\mu_{t}$, $t>0$ (respectively  $\tilde{\mu}_{t}$, $t>0$) has as density function the hazard function
$$
h_t(x)=\frac{1}{\mu_{t}(L)}1_{L}=\exp(t\nu(L^{c}))1_{L}
$$
(respectively
$$
\tilde{h}_t(x)=\frac{1}{\tilde{\mu}_{t}(\mathring{L}^{c})}1_{\mathring{L}^{c}}=\exp(t\nu(\mathring{L}))1_{\mathring{L}^{c}}
$$
).
\end{lemma}
\begin{proof}
	First, let us show that $\tilde{\mu}_{t}(\mathring{L})=T(\mathring{L})=1-\exp(-t\nu(\mathring{L}))$. If $L$ is a filter then the running maximum is not in $\mathring{L}$ if no jump falls in $\mathring{L}$ during the time span $[0,t]$. The probability that the running maximum is in the set $\mathring{L}$ is given by:
	\begin{align*}
		P(\vee_{s \leq t} \Delta X_s \in \mathring{L} )&=\tilde{\mu}_{t}(\mathring{L})\\
		&=1-P(\vee_{s \leq t} \Delta X_s \notin \mathring{L} )\\
		&=1-P(\Delta X_s\notin \mathring{L} , s \leq t)\\
		&=1-\exp(-t\nu(\mathring{L})).
	\end{align*}
On the other hand, the running minimum is in $L$ if all the jumps fall in $L$;
\begin{align*}
	P(\wedge_{s \leq t} \Delta X_s \notin L )&=\mu_{t}(L^{c})\\
	&=1-P(\wedge_{s \leq t} \Delta X_s \in L )\\
	&=1-P(\Delta X_s \in L , s \leq t)\\
	&=1-\exp(-t\nu(L^{c})).
	\end{align*}

Since $L$ and $\mathring{L}^{c}$ are filters  respectively for $(\overline{\mathbb{R}}^{d}, *)$ and $(\overline{\mathbb{R}}^{d}, \tilde{*})$  with $L=J(l_1)\times \cdots \times J(l_d)$, the exponential famlies are $\{h_{t} \mu_t, t>0\}$ and $\{ \tilde{h}_{t} \tilde{\mu}_t, t>0\}$ with
$$
h_{t}(x)=\frac{1}{\mu_{t}(L)}1_{L}=\exp(t\nu(L^{c}))1_{L},
$$

$$
\tilde{h}_{t}(x)=\frac{1}{\tilde{\mu}_{t}(\mathring{L}^{c})}1_{\mathring{L}^{c}}=\exp(t\nu(\mathring{L}))1_{\mathring{L}^{c}}.
$$
\end{proof}
\begin{lemma}
Assume that $\mu_{t}$, $\tilde{\mu}_{t}$, $t>0$, $\nu$ have no atoms. Let $N_{t}$ (resp. $\tilde{N}_{t}$ ) be  the point process  giving the number of lower records (resp. upper records) that fall in a measurable set. If $H_t$  (resp. $\tilde{H}_t$) is the hazard measure associated with $\mu_t$ (resp. associated with $\tilde{\mu}_t$), and $S_{\mu_t}$ (resp. $S_{\tilde{\mu}_t}$ ) is the support of $\mu_t$ (resp. $\tilde{\mu}_t$), $H_{t}$ (resp. $\tilde{H}_t$) verifies
\begin{align}
E(N_{t}(L^{c}))&= H_{t}(L^{c})=t\nu(S_{\mu_t} \cap L^{c}) \nonumber \\
(\text{ resp.}&\\
E(\tilde{N}_{t}(\mathring{L}))&=\tilde{H}_t(\mathring{L})=t\nu(S_{\mu_t} \cap \mathring{L}) ). \nonumber 
\end{align}
\end{lemma}
\begin{proof}
We provide a demonstration for $\tilde{N}_{t}$ only; the same arguments applies for the process $N_{t}$. On filters, the measure $\tilde{\mu}_{t}(\cdot)$ coincides with the hitting functional $1-\exp\{-t\nu(\cdot)\}$. For a filter $L$, the hazard measure evaluated at $\mathring{L}$  is:
\begin{align*}
\tilde{H}_{t}(\mathring{L})&=\int_{S_{\tilde{\mu}_{t}}\cap \mathring{L}}\frac{\tilde{\mu}_{t}(dl)}{1-\tilde{\mu}_{t}(\mathring{J}(l_1) \times \cdots \times \mathring{J}(l_d))}\\
&=\int_{S_{\tilde{\mu}_{t}}\cap \mathring{L}}\frac{d(1-\exp(-t\nu(\mathring{J(l_1)} \times \cdots \times \mathring{J(l_d)})))}{\exp(-t\nu(\mathring{J(l_1)} \times \cdots \times \mathring{J(l_d)})}\\
		&=t\int_{S_{\tilde{\mu}_{t}}\cap \mathring{L}}\nu(dl)\\
		&=t\nu(S_{\tilde{\mu}_{t}}\cap \mathring{L}).
\end{align*}

Since $\tilde{H}_t$  is the hazard measure associated with $\tilde{N}_{t}$,  we have $E(\tilde{N}_{t}(\mathring{L})= \tilde{H}_t(\mathring{L})$. The proof uses the same arguments as
\citet[proof of theorem 2.1]{goldie1989records}. Let us recall that  $\tilde{X}_{n}(t)=(\tilde{X}_{n,1}(t), \ldots,\tilde{X}_{n,d}(t))$ is a record if, for $k<n$, $\tilde{X}_{k}(t) \in \mathring{J}(\tilde{X}_{n,1}(t)) \times\cdots \times \mathring{J}(\tilde{X}_{n,d}(t))$. Let us call $A_n$ the event $
A_n=\{\tilde{X}_{n}(t) \text{ is a record}\}$,$n \geq 1$. Given $\tilde{X}_{n}(t)=l$, the event $A_n$ is realized if and only if, for  $k<n$, $\tilde{X}_{k}(t) \in \mathring{J}(l_1) \times \cdots \times \mathring{J}(l_d)$. The expected number of upper records in $\mathring{L}$ is

	\begin{align*}
	E(\tilde{N}_{t}(\mathring{L}))&=E(\sum_{n=1}^{+\infty}1_{A_n})\\
	&=\sum_{n=1}^{+\infty}P(A_n)\\	
	&=\sum_{n=1}^{+\infty} \int_{S_{\tilde{\mu}_{t}}\cap \mathring{L}} \tilde{\mu}_{t}(\mathring{J(l_1)} \times \cdots \times \mathring{J(l_d)})^{n-1} d\tilde{\mu}_{t}(l)\\
	&= \int_{S_{\tilde{\mu}_{t}}\cap \mathring{L}} \sum_{n=1}^{+\infty}\tilde{\mu}_{t}(\mathring{J}(l_1) \times \cdots \times \mathring{J}(l_d))^{n-1} d\tilde{\mu}_{t}(l)\\
	&=\int_{S_{\tilde{\mu}_{t}}\cap \mathring{L}}  \frac{d\tilde{\mu}_{t}(l)}{1-\tilde{\mu}_{t}(\mathring{J}(l_1) \times \cdots \times \mathring{J}(l_d))}\\
	&=\tilde{H}_{t}(\mathring{L}).
\end{align*}		
\end{proof}

The hitting functional does not define a probability measure because it is not additive. However, it is possible to construct a probability measure based on the hitting functional by choosing appropriately the sigma-algebra. We will consider the Borel sigma-algebra generated by filters (see theorem \ref{theorem1}). Lemma \ref{measurability} assures that the space is Hausdorff and proves that the support of the upper and lower record distributions is $\mathbb{R}^{d}\setminus \{(0,\ldots,0)\}$. Lemma \ref{measurability} and theorem \ref{theorem1} rely on the following hypotheses.\\

{\bf Hypotheses.}

\begin{itemize}
	\item [($\mathcal{H}_1$)] The jumps of size $|\Delta X_{t}|=0$, $t>0$ are not considered in the record process:
$$
\mu_t(0,\ldots,0)=\tilde{\mu}_t(0,\ldots,0)=0;
$$
	\item [($\mathcal{H}_2$)]  the L\'{e}vy measure does not charge events where for some coordinates, the process of records is in a cemetery state:
	$$
	\nu(\cup_{i=1}^{d}\{y=(y_1, \ldots, y_d) \in \overline{\mathbb{R}}^{d}:y_{i} \in\{-\infty,+\infty\}\})=0.\\
	$$
		\item [($\mathcal{H}_3$)]  the L\'{e}vy process has infinite activity in the neighbourhoods of $(0,\ldots,0)$
	$$
	\nu(\overline{\mathbb{R}}^{\epsilon_1}\times \cdots \times\overline{\mathbb{R}}^{\epsilon_d}\setminus \{(0,\ldots,0)\})=+\infty,
	$$
with $\epsilon_i \in \{-,+\}$, $i \in \{1, \ldots,d\}$.
\end{itemize}

These three hypotheses will guarantee the existence of a measure that coincides with the hitting functional on filters and have $\nu$ as hazard measure.
\begin{lemma}[Monotone class and measurability]
\label{measurability}
	The following results characterize the sigma-algebra  generated by the filters and the topology generated by the partial order:
	\begin{itemize}
		\item [(a)] the set of filters $\mathcal{L}$,
		\begin{align*}
			\mathcal{L}=\{L=J(l_1) \times \cdots \times J(l_d):
			&J(l_i)=[l_i, +\infty]  \text{ if } l_i \geq 0,\\
			&J(l_i)=[-\infty, l_i], \text{ if } l_i < 0, i \in \{1, \ldots,d\} \}
		\end{align*}
		is a monotone class: it is closed under the countable union of increasing sequences of sets and the countable intersection of decreasing sequences of sets;
		\item [(b)] the graph of the relation $\leq$, $G=\{(x,y)\in \overline{\mathbb{R}}^{d}\times \overline{\mathbb{R}}^{d}: x \leq y\}$ is $\sigma(\mathcal{L})\times \sigma(\mathcal{L})$-measurable and is a closed set. The sets  $J(l_1)\times \cdots \times J(l_d)$ and $\mathring{J}(l_1)\times \cdots \times \mathring{J}(l_d)$  are measurable for any $(l_1, \ldots, l_{d}) \in \mathbb{R}^{d}$;
		\item [(c)] on $\overline{\mathbb{R}}^{\epsilon_1}\times \cdots \times \overline{\mathbb{R}}^{\epsilon_d}$, $\epsilon_i \in \{-,+\}$, the support of $\tilde{\mu}_{t}$, the law of the upper record of the jumps, is the set 
		$$
		S_{\tilde{\mu}_{t}}=\{l=(l_1, \ldots, l_d):\tilde{\mu}_{t}(\mathring{J}(l_1) \times \cdots \times \mathring{J}(l_d))<1 \}=\mathbb{R}^{\epsilon_1}\times \cdots \times \mathbb{R}^{\epsilon_d}\setminus \{(0,\ldots,0)\};
		$$
		\item [(d)] on $\overline{\mathbb{R}}^{\epsilon_1}\times \cdots \times \overline{\mathbb{R}}^{\epsilon_d}$, $\epsilon_i \in \{-,+\}$, the support of $\mu_{t}$, the law of the lower record of the jumps, is the set 
		$$
		S_{\mu_{t}}=\{l=(l_1, \ldots, l_d):\mu_{t}(J(l_1) \times \cdots \times J(l_d))>0\}=\mathbb{R}^{\epsilon_1}\times \cdots \times \mathbb{R}^{\epsilon_d}\setminus \{(0,\ldots,0)\}.
		$$
	\end{itemize}
\end{lemma}

\begin{proof}
		\begin{itemize}
		\item [(a)] Since the intersection of filters is again a filter, the set of filters $\mathcal{L}$ is closed under the countable intersection of increasing sequences of sets. If $L_{1}, L_{2}$ are two filters such that $L_{1} \subset L_{2}$, with  $L_{1}=J(l_{1,1}) \times \cdots \times J(l_{1,d})$, $L_{2}=J(l_{2,1}) \times \cdots \times J(l_{2,d})$, then $l_1=(l_{1,1}, \ldots, l_{1,d})\geq l_2=(l_{2,1}, \ldots, l_{2,d})$. Assume that $\{l_n: n \geq 1\}$ is a decreasing sequence, the corresponding set of filters $\{L_n: n \geq 1\}$ is an increasing sequence and $\cup_{n \geq 1}L_n=J(l_{1}) \times \cdots \times J(l_{d})$ with $l=(l_{1}, \ldots,l_{d})=\wedge_{n\geq 1}l_n=\inf_{n\geq 1} l_n$.
		
		\item [(b)] If $(x,y)\in G$ then $y \in J(x_{1}) \times \cdots \times J(x_{d})$, $x=(x_{1}, \ldots, x_{d})$ or $x \notin \mathring{J}(y_{1}) \times \cdots \times \mathring{J}(y_{d})$, $y=(y_{1}, \ldots, y_{d})$. Let $(x,y)$ be an accumulation point of a sequence $(x_{n},y_{n})_{n \geq 1}$  of elements of of $G$; there exists a monotone subsequence $(x_{k(n)})_{n\geq 1}$ of elements converging to $x$. Let us suppose that the subsequence  $(x_{k(n)})_{n\geq 1}$ is non-decreasing; we have $y \in \cap_{n \geq 1}J(x_{k(n),1}) \times \cdots \times J(x_{k(n),d})$ and $x \leq y$. If, in the contrary the  subsequence  $(x_{k(n)})_{n\geq 1}$ is non-increasing, then  we have $y \in \cup_{n \geq 1}J(x_{k(n),1}) \times \cdots \times J(x_{k(n),d})$ and $x \leq y$. The set $G$ is a Borel set in the product topology and is measurable.

		\item [(c)] We have to show that $\tilde{\mu}_{t}(S_{\tilde{\mu}_{t}})=1$. Relying on the equivalence
		$$
		1-\tilde{\mu}_{t}(\mathring{J(l_1)} \times \cdots \times \mathring{J(l_d)})>0 \iff \nu(\mathring{J}(l_1) \times \cdots \times \mathring{J}(l_d))< +\infty,
		$$
	we can affirm that if $l \in S_{\tilde{\mu}_{t}}$ then any $x \geq l$	is in $S_{\tilde{\mu}_{t}}$ since $\mathring{J}(x_1) \times \cdots \times \mathring{J}(x_d)$ is a subset of $\mathring{J}(l_1) \times \cdots \times \mathring{J}(l_d)$ which have a finite L\'{e}vy measure. Using the hypotheses $\mathcal{H}_1$ and $\mathcal{H}_2$, we can deduce that 
$$
	S_{\tilde{\mu}_{t}}  \supset  \cup_{l \in \mathbb{R}^{d}\setminus \{(0,\ldots,0)\}} \mathring{J}(l_1)\setminus \{-\infty,+\infty\} \times \cdots \times \mathring{J}(l_d)\setminus \{-\infty,+\infty\}\\
$$
since $\nu(\mathring{J}(l_1)\setminus \{-\infty,+\infty\} \times \cdots \times \mathring{J}(l_d)\setminus \{-\infty,+\infty\})< +\infty$ for $l \in \mathbb{R}^{d}\setminus \{(0,\ldots,0)\}$. Conversely, we have 	$	S_{\tilde{\mu}_{t}}  \subset \mathbb{R}^{d}\setminus \{(0,\ldots,0)\}$ since
\begin{align*}
\tilde{\mu}_{t}(\overline{\mathbb{R}}^{d}\setminus (\mathbb{R}^{d}\setminus \{(0,\ldots,0)\}))&= \tilde{\mu}_{t}((0,\ldots,0))+ \tilde{\mu}_{t}(\cup_{i=1}^{d}\{y=(y_1, \ldots, y_d) \in \overline{\mathbb{R}}^{d}:y_{i} \in\{-\infty,+\infty\}\})\\
&=\tilde{\mu}_{t}(\cup_{i=1}^{d}\cup_{(\epsilon_1, \ldots,\epsilon_d) \in\{-1,1\}^{d}}\{y=(y_1, \ldots, y_d) \in \overline{\mathbb{R}}^{d}:y_{i} \in\{-\infty,+\infty\},\\
& sgn(y_j)=\epsilon_j, j \in \{1,\ldots,d\}\})\\
& \leq \sum_{i=1}^{d}\sum_{(\epsilon_1, \ldots,\epsilon_d) \in\{-1,1\}^{d}}\tilde{\mu}_{t}(\{y=(y_1, \ldots, y_d) \in \overline{\mathbb{R}}^{d}:y_{i} \in\{-\infty,+\infty\},\\ 
& sgn(y_j)=\epsilon_j, j \in \{1,\ldots,d\}\})\\
& \leq \sum_{i=1}^{d}\sum_{\epsilon_j\in\{-1,1\}, j \neq i, 1 \leq j \leq d}\big \{1-\exp(-t\nu(\{y=(y_1, \ldots, y_d) \in \overline{\mathbb{R}}^{d}:\\
&y_{i} \in\{-\infty,+\infty\}, sgn(y_j)=\epsilon_j, j \in \{1,\ldots,d\}\})) \big \} \\
& \leq 0
\end{align*}	
As a consequence, we have
$$
1=\tilde{\mu}_{t}( \mathbb{R}^{d}\setminus \{(0,\ldots,0)\})=\tilde{\mu}_{t}(S_{\tilde{\mu}_{t}}).
$$
\item [(d)] The proof for $S_{\mu_{t}}$, $t>0$, uses the same arguments as in $(c)$. 
	\end{itemize}
\end{proof}	

\begin{theorem}
\label{theorem1}
Assume that $\mu_{t}$, $\tilde{\mu}_{t}$, $t>0$, $\nu$ have no atoms. Given a distribution of  the lower record  of the jumps $\mu_{t}$ (respectively a distribution of the upper record $\tilde{\mu}_{t}$ ), the hazard measure 
\begin{align}
\label{eqhazard}
 A \mapsto	H_{t}(A)=\int_{S_{\mu_{t}}\cap A}\frac{1}{\mu_{t}( J(l_1)\times \cdots \times J(l_d))} \mu_{t}(dl),
\end{align}

(respectively
\begin{align}
\label{eqhazard2}
 A \mapsto \tilde{H}_{t}(A)=\int_{S_{\tilde{\mu}_{t}} \cap A} \frac{1}{1-\tilde{\mu}_{t}( \mathring{J(l_1)}\times \cdots \times  \mathring{J(l_d)})} \tilde{\mu}_{t}(dl),
\end{align}
) coincides with the  L\'{e}vy  measure on filters (respectively on the complement of filters).

The equations (\ref{eqexpo}) and (\ref{eqhazard}) (respectively (\ref{eqexpo}) and (\ref{eqhazard2})) establish a bijection between the set of the   distributions of the lower record of the jumps (respectively the upper record of the jumps) and the set of L\'{evy} measures.
\end{theorem}

\begin{proof}
Let us define a probability measure $P^{*}$ as follows. If $L_1$ and $L_2$ are filters, $S_1$, $S_2$ are sets such that $S_1 \in \{\mathring{L}_{1}, \mathring{L}_1^{c}\}$, $S_2 \in \{\mathring{L}_2, \mathring{L}_2^{c}\}$, the probability of  $S_1 \cup S_2$ is defined by
$$
P^{*}(S_1 \cup S_2)=P^{*}(S_1)+P^{*}(S_2)-P^{*}(S_1 \cap S_2)
$$
with
$$
P^{*}(\mathring{L})=1-\exp(-t\nu(\mathring{L})) \text{ if } L \text{ is a filter. }
$$

$P^{*}$ is additive and well defined on the algebra $\mathcal{A}(\mathcal{L})$ generated by the set of filters $\mathcal{L}$. The probability of the events $\mathring{L}_1^{c}, \mathring{L}_1 \cup \mathring{L}_2, \mathring{L}_1 \cup \mathring{L}_2^{c}, \mathring{L}_1 \cap \mathring{L}_2, \mathring{L}_1 \cap \mathring{L}_2^{c} $ are given by the following formulas:
\begin{align*}
P^{*}(\mathring{L}_1 \cup \mathring{L}_2)&=1-\exp(-t\nu(\mathring{L}_1))-\exp(-t\nu(\mathring{L}_2))+\exp(-t\nu(\mathring{L}_1 \cap \mathring{L}_2));\\
P^{*}(\mathring{L}_1^{c})&=P^{*}(\mathring{L}_1 \cup \mathring{L}_1^{c})-P^{*}(\mathring{L}_1)=\exp(-t\nu(\mathring{L}_1));\\
P^{*}(\mathring{L}_1 \cap \mathring{L}_2)&=1-\exp(-t\nu(\mathring{L}_1 \cap \mathring{L}_2)) \text{ since } \mathring{L}_1 \cap \mathring{L}_2 \text{ is still a filter; }\\
P^{*}(\mathring{L}_1 \cap \mathring{L}_2^{c})&=P^{*}(\mathring{L}_1)-P^{*}(\mathring{L}_1 \cap \mathring{L}_2)=\exp(-t\nu(\mathring{L}_1 \cap \mathring{L}_2))-\exp(-t\nu(\mathring{L}_1));\\
P^{*}(\mathring{L}_1 \cup \mathring{L}_2^{c})&=1-P^{*}(\mathring{L}_1^{c} \cap \mathring{L}_2)=1-\exp(-t\nu(\mathring{L}_1 \cap \mathring{L}_2))+\exp(-t\nu(\mathring{L}_2)).
\end{align*}
The hitting functional verifies the FKG inequality,
$$
T(\mathring{L}_1 \cup \mathring{L}_2)+T(\mathring{L}_1 \cap \mathring{L}_2) \geq  T(\mathring{L}_1 ) +T( \mathring{L}_2),
$$
which means that $P^{*}(\mathring{L}_1 \cup \mathring{L}_2) \leq T(\mathring{L}_1 \cup \mathring{L}_2) \leq 1$. On the other hand, since 
$\mathring{L}_1 \cap \mathring{L}_2 \subset \mathring{L}_2$, we have $\exp(-t\nu(\mathring{L}_1 \cap \mathring{L}_2)) -\exp(-t\nu( \mathring{L}_2)) \geq 0$ and $P^{*}(\mathring{L}_1 \cup \mathring{L}_2) \geq 0$. Using similar arguments, it is easy to show that  $P^{*}(\mathring{L}_1^{c})$, $P^{*}(\mathring{L}_1 \cap \mathring{L}_2)$, $P^{*}(\mathring{L}_1 \cap \mathring{L}_2^{c})$, $P^{*}(\mathring{L}_1 \cup \mathring{L}_2^{c})$ are all in the interval $[0,1]$. $P^{*}$ and $\tilde{\mu}_{t}$ coincide on the algebra $\mathcal{A}(\mathcal{L})$ generated by the set of filters $\mathcal{L}$ which is a monotone class; $P^{*}$ and $\tilde{\mu}_{t}$ coincide then on the sigma-algebra $\sigma(\mathcal{L})$ generated by $\mathcal{L}$. The hazard measure associated with $P^{*}$ coincides with $t \nu$ on the filters. Using again the monotone class argument, we can say that $\tilde{H}_{t}$ coincides with $t \nu$ on the sigma-algebra $\sigma(\mathcal{L})$. The formula $\tilde{H}_{t}(dl)=\tilde{\mu}_{t}(dl)/(1-\tilde{\mu}_{t}( \mathring{J(l_1)}\times \cdots \times  \mathring{J(l_d)}))$ defines a one-to-one correspondence between $\tilde{\mu}_{t}$ and  $(1/t)\tilde{H}_{t}=\nu$ on $\sigma(\mathcal{L})$: the inverse correspondence is the application $\nu \mapsto P^{*}$.
\end{proof}
In the lemma (\ref{exponent}), the expression for the cumulant function
is  $Cu(\chi)=\log(\mu_t(L))=-t \nu(L)$ for $\mu_t$ and $\tilde{Cu}(\chi)=\log(\tilde{\mu}_t(\mathring{L}^{c}))=-t \nu(\mathring{L})$ for $\tilde{\mu}_t$, $t>0$. 
\begin{corollary}
Let $(\mu_{t})_{t\geq 0}$ (resp. $(\tilde{\mu}_{t})_{t\geq 0}$) be a convolution semi-group of probability measures. If the semi-group of semi-characters $\chi=1_{L}$ separates points, then there exists a unique L\'{e}vy measure $\nu$ such that the Laplace transform of $\mu_{t}$ (resp. $\tilde{\mu}_{t}$) is given by
\begin{align}
\label{cor01}
E_{\mu_t}(1_{L^{c}}(\wedge_{s \leq t} \Delta X_{s}))=\exp(t Cu(\chi))=\exp(-t \nu(L^{c}))
\end{align}
(resp.
\begin{align}
\label{cor02}
E_{\tilde{\mu}_t}(1_{\mathring{L}^{c}}(\vee_{s \leq t} \Delta X_{s}))=\exp(t \tilde{Cu}(\chi))=\exp(-t  \nu(\mathring{L}))
\end{align}
). Conversely, given a L\'{e}vy measure $\nu$, there exists a convolution semi-group of probability measures $(\mu_{t})_{t\geq 0}$ (resp. $(\tilde{\mu}_{t})_{t\geq 0}$) such that (\ref{cor01}) (resp. (\ref{cor02}) ) holds true.
\end{corollary}
 \begin{proof}
 The equations \ref{cor01} and \ref{cor02} are a consequence of lemma \ref{exponent}. The unicity is assured by the injectivity of the Laplace transform.
 \end{proof}
\subsection{Properties of the law of the upper and lower record}
The distribution of the upper record is not strictly max-infinitely divisible, but exhibits most of the properties of max-infinitely divisible distributions. We will adopt the following definitions in order to cope with this situation.  
 \begin{definition}
 	A distribution function $G$ on $\mathbb{R}^{d}$ will be said max-infinitely divisible (max-id) if, for every integer $k>1$, there exists a cdf $G_{k}$  and a continuous function $h$ such that $h(G(x))=(h(G_{k}(x)))^{k}$ for every $x \in\mathbb{R}^{d}$.
 \end{definition}
  \begin{definition}
 	A distribution function $G$ on $\mathbb{R}^{d}$ will be said max-stable if, for every $t>0$, there exists a continuous function $h$ such that 
 	$$h(G(t x))^{t}=h(G(x)) $$ for every $x \in\mathbb{R}^{d}$.
 \end{definition}

\begin{proposition}
The distributions of the upper  and lower record are max-infinitely divisible  distributions. In addition, if the L\'{e}vy measure satisfies $ \nu(\lambda^{-1} A)=\lambda\nu(A )$ for any measurable set $A$, then the distributions of the upper and lower record are max-stable.
\end{proposition}
\begin{proof}
Using lemma \ref{exponent}, we can see that the law of the lower record is  max-infinitely divisible and max-stable for $h(x)=x$, $x \in [0,1]$ and the law of the upper record is  max-infinitely divisible and max-stable for $h(x)=1-x$, $x \in [0,1]$.
\end{proof}

Infinitely divisible  distributions  are often characterized in term of  their L\'{e}vy measure; for  example, for $d=1$, the class of self-decomposable distributions are the  infinitely  divisible distributions for which the L\'{e}vy measure is  of the form
 $ \nu(dx)=(q(x)/|x|)dx$, $q$ is increasing on $(-\infty,0)$ and decreasing on $(0, +\infty)$. In a dimension $d \geq 2$, the characterization of infinitely divisible  distributions  is done through the radial component of a polar decomposition of the L\'{e}vy measure \citep{barndorff2006some}. Let us consider the following polar decomposition of the L\'{e}vy measure
$$
\nu(A)=\int_{\mathbb{S}_{d}} (\int_{[0,+\infty)} 1_{A}(r \zeta) \nu_{\zeta}(dr))\lambda(\zeta).
$$
for any Borel set $A \in \mathbb{R}^{d}$, $\lambda$ a positive measure on $\mathbb{S}_{d}=\{x: \lVert x \rVert=1\}$; $\nu_{\zeta}([r, +\infty))$ is measurable in $\zeta$ for any set $[r, +\infty)$, $r \geq 0$. We have the following characterizations of infinitely divisible  distributions \citep{barndorff2007levy, barndorff2006some}:
\begin{itemize}
\item [(i)] the  self-decomposable distributions  correspond to   $r\nu_{\zeta}(dr)/dr$ is measurable and decreasing on $]0,+\infty[$ for $\lambda-$almost all $\zeta$ ;
\item [(ii)] the Bondesson class, which contains mixtures of exponential distributions, corresponds to $\nu_{\zeta}(dr)/dr$ completely monotone for $\lambda-$almost all $\zeta$
\item [(iii)] the Thorin class which contains Gamma distributions corresponds to  $r \nu_{\zeta}(dr)/dr$  measurable and completely monotone for $\lambda-$almost all $\zeta$.
\end{itemize}
For a given $\zeta$, let us define the distribution of the upper and lower for the radial component:
\begin{align*}
\mu_{t,\zeta}([r,+\infty])&=\exp(-t\nu_{\zeta}([0,r))),\\
\tilde{\mu}_{t,\zeta}((r,+\infty])&=1-\exp(-t\nu_{\zeta}((r,+\infty])), r \geq 0.
\end{align*}
\begin{proposition}
\label{prop32}
The class of self-decomposable distributions (resp. the Thorin class) is characterized by a function $r \mapsto r\mu_{t,\zeta}(dr)/(\mu_{t,\zeta}dr)$ decreasing (resp. completely monotone) on $]0,+\infty[$ for $\lambda-$almost all $\zeta$. The Bondesson class corresponds to  $\mu_{t,\zeta}(dr)/(\mu_{t,\zeta}dr)$ completely monotone for $\lambda-$almost all $\zeta$.
\end{proposition}
\begin{proof}
From the definition of $\mu_{t,\zeta}$, we have $\nu_{\zeta}(dr)=\mu_{t,\zeta}(dr)/(\mu_{t,\zeta}dr)$. By replacing $\nu_{\zeta}(dr)$ by its equivalent expression in the definition of the different classes, we have the result.
\end{proof}
\section{L\'{e}vy copulas induced by the upper and lower record}
\subsection{Distribution function defined by the hitting functional}
\citet{barndorff2007levy} defined positive L\'{e}vy copulas by considering the image of the L\'{e}vy measure $\nu$ under the transformation $(x_1, \ldots, x_d) \mapsto (1/x_1, \ldots, 1/x_d)$. We extend this framework to $\overline{\mathbb{R}}^{\epsilon_1}\times \cdots \times \overline{\mathbb{R}}^{\epsilon_d}$, with $\epsilon_i \in \{+,-\}$, $i=1, \ldots,d$ by considering the image of the L\'{e}vy measure by the application $(x_1, \ldots, x_d) \mapsto (y_1, \ldots, y_d)$, with $y_i=x_i$ if $x_i<0$, $y_i=1/x_i$ if $x_i\geq 0$. On  $\overline{\mathbb{R}}^{\epsilon_1}\times \cdots \times \overline{\mathbb{R}}^{\epsilon_d}$, let us define the applications
\begin{align*}
G_{u,\epsilon_1 \cdots \epsilon_d}:&  (x_{1}, \ldots,x_{d}) \mapsto \tilde{\mu}_{t}(J^{+}(x_1) \times \cdots \times J^{+}(x_d)) \\
G_{l,\epsilon_1 \cdots \epsilon_d}:&  (x_{1}, \ldots,x_{d}) \mapsto \mu_{t}(J^{+}(x_1) \times \cdots \times J^{+}(x_d))
\end{align*}
where $J^{+}(x_i)=[1/x_i, +\infty]$ if $x_i \geq 0$ and $J^{+}(x_i)=[-\infty, x_ {i}]$ if $x_i < 0$, $i \in \{1, \ldots,d\}$. The functions $G_{u, \epsilon_1 \cdots \epsilon_d}$ and $G_{l, \epsilon_1 \cdots \epsilon_d}$ were introduced in the proof of lemma \ref{pre1}.
\begin{proposition}
$G_{u, \epsilon_1 \cdots \epsilon_d}$ and $G_{l, \epsilon_1 \cdots \epsilon_d}$ are cdfs. If $1-G_{u,\epsilon_1 \cdots \epsilon_d}(x)=\prod_{i=1}^{d} (1-G_{i,\epsilon_i}(x_{i}))$ (resp. $G_{l,\epsilon_1 \cdots \epsilon_d}(x)=\prod_{i=1}^{d} G_{i,\epsilon_i}(x_{i})$) where $G_{i,\epsilon_i}$  is a distribution on $\overline{\mathbb{R}}^{\epsilon_i}$, $1\leq i \leq d$, then 
$
\nu(\cup_{1 \leq i_{1} < i_{2} \leq d}\{y: y_{i_1} \neq 0, y_{i_2} \neq 0 \})=0.
$
\end{proposition}
\begin{proof}
Both functions $G_{u,\epsilon_1 \cdots \epsilon_d}$ and $G_{l,\epsilon_1 \cdots \epsilon_d}$ can be expressed in the following way:
\begin{align*}
G_{u,\epsilon_1 \cdots \epsilon_d}(x_1, \ldots, x_{d})&=E_{\tilde{\mu}_{t}}(1_{J^{+}(x_1) \times \cdots \times J^{+}(x_d)}),
G_{l,\epsilon_1 \cdots \epsilon_d}(x_1, \ldots, x_{d})&=E_{\mu_{t}}(1_{J^{+}(x_1) \times \cdots \times J^{+}(x_d)}).
\end{align*}
The  function  $G_{u,\epsilon_1 \cdots \epsilon_d}$ and $G_{l,\epsilon_1 \cdots \epsilon_d}$ are $d-$increasing since the application 
$$x \mapsto 1_{J^{+}(x_1) \times \cdots \times J^{+}(x_d)}$$
 is $d-$increasing; $G_{u,\epsilon_1 \cdots \epsilon_d}$ (resp. $G_{l,\epsilon_1 \cdots \epsilon_d}$) is equal to $0$ at the bottom of $\overline{\mathbb{R}}^{\epsilon_1}\times \cdots \times \overline{\mathbb{R}}^{\epsilon_d}$ (see assumption $\mathcal{H}_2$) and is equal to $1$ at the top of $\overline{\mathbb{R}}^{\epsilon_1}\times \cdots \times \overline{\mathbb{R}}^{\epsilon_d}$ (see assumption $\mathcal{H}_3$ and the fact that $\tilde{\mu}_{t}(\mathring{L})=1-\exp(-t \nu(\mathring{L}))$). If $1-G_{u,\epsilon_1 \cdots \epsilon_d}$ is a product of survival functions
$1-G_{u, \epsilon_1 \cdots \epsilon_d}=\prod_{i=1}^{d} (1-G_{i,\epsilon_i})$	 then for distinct $i_1, i_2$,  and $\epsilon_1=\cdots=\epsilon_d=-$, we have
\begin{align}
	\label{eqsupport}
1-G_{u,- \cdots -}(+\infty, \ldots, x_{i_1},+\infty, \ldots,+\infty, x_{i_2},+\infty, \ldots) 
&=\exp{\{-t \nu(\{y: y_{i_1} \leq x_{i_1}, y_{i_2}\leq x_{i_2}\})\}} \nonumber \\
&=(1-G_{i_1,-}(x_{i_1}))(G_{i_2,-}(x_{i_2})) \nonumber \\
&=\prod_{i\in \{i_1, i_2\}}\exp{\{-t \nu(\{y:  y_{i} \leq x_{i} \})\}}.
\end{align}
The equations (\ref{eqsupport}) are satisfied for all $t>0$ if and only if
\begin{align}
\label{eqsupport2}
\nu(\{y: y_{i_1} \leq x_{i_1}, y_{i_2}\leq x_{i_2} \})&=\nu(\{y:  y_{i_1} \leq x_{i_1} \}) + \nu(\{y: y_{i_2}\leq x_{i_2}\}) \\
&\geq \min(\nu(\{y: y_{i_1} \leq x_{i_1} \}) , \nu(\{y: y_{i_2}\leq x_{i_2}\})). \nonumber 
\end{align}
Since, by definition, the set $\{y: y_{i_1} \leq x_{i_1}, y_{i_2}\leq x_{i_2} \}$ is a subset of the sets $\{y: y_{i_1} \leq x_{i_1} \}$ and $\{y: y_{i_2}\leq x_{i_2}\}$, we have  in fact
$$
\nu(\{y: y_{i_1} \leq x_{i_1}, y_{i_2}\leq x_{i_2}\}) = \min(\nu(\{y: y_{i_1} \leq x_{i_1} \}) , \nu(\{y:  y_{i_2}\leq x_{i_2}\}));
$$
as a result, the greater term on the right side of equation (\ref{eqsupport2}), 
$$ \max(\nu(\{y: y_{i_1} \leq x_{i_1} \}) , \nu(\{y: y_{i_2}\leq x_{i_2}\})),$$
is equal to $0$  for all $x_{i_1}<0, x_{i_2}<0$. Then, we have
$$
\nu(\{y: y_{i_1} \leq x_{i_1}, y_{i_2}\leq x_{i_2}\})=0 \text{ for } x_{i_1}<0, x_{i_2}<0
$$
which means that 
$$
\nu(\{y: y_{i_1} \neq 0, y_{i_2} \neq 0\})=0.
$$
The arguments are the same for any $(\epsilon_1, \ldots, \epsilon_d) \in \{-,+\}^{d}$. The cdf $G_{l,\epsilon_1 \cdots \epsilon_d}$ is a max-id cdf under the product order:
$$
G_{l,\epsilon_1 \cdots \epsilon_d}(x_1, \ldots,x_d)=\exp(-t\nu(\overline{\mathbb{R}}^{\epsilon_1}\times \cdots \times \overline{\mathbb{R}}^{\epsilon_d} \setminus J^{+}(x_1) \times \cdots \times J^{+}(x_d))).
$$
The components are independent if and only if the measure $\nu$ concentrates on axes \citep[Prop.5.24]{resnick2013extreme}.
\end{proof}

Let us recall that a L\'{e}vy copula is defined as a function linking a tail integral function to its marginal tail integral functions. In section 2, we defined three tail integral functions:

\begin{itemize}
	\item [(i)] the tail integral function $U: (\mathbb{R}\setminus\{0\})^{d} \rightarrow \mathbb{R}$ 
	\begin{align*}
	U(x_1, \ldots, x_d )=(\prod_{i=1}^{d}\epsilon_i)  \nu(I(x_1) \times I(x_2) \times \cdots \times I(x_d))
	\end{align*}
	where $I(x_i)=[x i , + \infty)$, $\epsilon_i=1$ if $x_i \geq 0$, $I(x_i)=(-\infty, x_i)$, $\epsilon_i=-1$ if $x_i<0$, $i=1, \ldots,d$;   $U$ can take negative values; 
	\item [(ii)] the non-negative tail integral functions 	$U_u^{+}$, and $U_l^{+}$
\begin{align*}
U_{u}^{+}(x_1, \ldots, x_d )&=  \nu(J^{+}(x_1) \times J^{+}(x_2) \times \cdots \times J^{+}(x_d)),\\
U_{l}^{+}(x_1, \ldots, x_d )&=-\log(1-\exp(- \nu(\overline{\mathbb{R}}^{\epsilon_1}\times\cdots\times  \overline{\mathbb{R}}^{\epsilon_d})\setminus J^{+}(x_1) \times J^{+}(x_2) \times \cdots \times J^{+}(x_d)).
\end{align*}	
	where $J^{+}(x_i)=[1/x i , + \infty]$, $\epsilon_i=+$ if $x_i \geq 0$ and $J^{+}(x_i)=[-\infty, x_i]$, $\epsilon_i=-$ if $x_i<0$, $i=1, \ldots,d$.
\end{itemize}
The following lemma expresses the relationship between the traditional tail integral function $U$ and the functions $U_{l}^{+}$,  $U_{u}^{+}$. 
\begin{lemma}
Let us suppose that $\nu$ has no atoms; we have:
\begin{align}
\label{sect4U}
 tU(x)    &=-(\prod_{i=1}^{d}sgn(x_i)\log(1-G_{u,\epsilon_1 \cdots \epsilon_d}(x_1^{-sgn(x_1)}, \ldots,x_d^{-sgn(x_d)} )),
\end{align}
where $sgn(x_i)=+1$ if $x_i \geq 0$ and $sgn(x_i)=-1$ if $x_i < 0$, $i=1, \ldots,d$.
\end{lemma}
\begin{proof}
If $\nu$ has no atoms then the function $U$ is continuous and we have 
\begin{align*}
	U(x_1, \ldots, x_d )=(\prod_{i=1}^{d}\epsilon_i)  \nu(\overline{I(x_1) \times I(x_2) \times \cdots \times I(x_d)}),
\end{align*}
$(x_1, \ldots, x_d ) \in \mathbb{R}^{d}$. The result comes from the expression of $U_{u}^{+}$, $U_{l}^{+}$ in terms of $\nu$ and the lemma \ref{exponent}.
\end{proof}
\begin{lemma}
There exists a proper copula $C_{u,\epsilon_1 \cdots \epsilon_d}$ defined on $\overline{Ran(G_{1,\epsilon_1})\times \cdots \times Ran(G_{d,\epsilon_d})}$ such that 
\begin{align*}
G_{u,\epsilon_1 \cdots \epsilon_d}(x_1, \ldots, x_d)=C_{u,\epsilon_1 \cdots \epsilon_d}(G_{1,\epsilon_1}(x_1), \ldots, G_{d,\epsilon_d}(x_d))
\end{align*}
In addition, the L\'{e}vy copula $F$ (respectively $F^{+}$) verifies the following formula:
\begin{align}
\label{eqsec411}
t F(-sgn(x_1)\frac{1}{t}\log(1-G_{1,\epsilon_1}(x_1^{-sgn(x_1)})),& \ldots, -sgn(x_d)\frac{1}{t} \log(1-G_{d,\epsilon_d}(x_d^{-sgn(x_d)}))=-\prod_{i=1}^{d}sgn(x_i) \nonumber\\
& \log\{1-C_{u, \epsilon_1 \cdots \epsilon_d}( G_{1,\epsilon_1}(x_1^{-sgn(x_1)}),\ldots, G_{d,\epsilon_d}(x_d^{-sgn(x_d)}))\}
\end{align}
\begin{align*}
t F^{+}(-\frac{1}{t}\log\{1-G_{1,\epsilon_1}(x_1)\}, \ldots, -\frac{1}{t}\log\{1-G_{d,\epsilon_d}(x_d)\})=-\log\{1-C_{\epsilon_1 \cdots \epsilon_d}(G_{1,\epsilon_1}(x_1), \ldots,G_{d,\epsilon_d}(x_d))\}
\end{align*}
).
\end{lemma}
\begin{proof}
For the traditional L\'{e}vy copula, the result comes by applying the Sklar's theorem to the equation \ref{sect4U}. For the probabilistic copula, we have the following expressions for the positive tail integral functions:
$$ t U_{u}^{+}(x)= -\log(1-G_{u,\epsilon_1 \cdots \epsilon_d}(x));$$
$$ t U_{l}^{+}(x)= -\log(1-G_{l,\epsilon_1 \cdots \epsilon_d}(x)).$$
Again, by applying the Sklar's theorem, we have the result.
\end{proof}
\begin{theorem} [Mapping of a L\'{e}vy copula into a  proper copula and vice versa]
	Let $F$  be the L\'{e}vy copula associated with a L\'{e}vy process in $\mathbb{R}^{d}$. The L\'{e}vy copula $F$ defines a proper copula $C$ on $[0,1]^{d}$: 
	$$
	C(1-\exp(-|x_1|), \ldots,1- \exp(-|x_d|))=1-\exp(-F(|x_1|, \ldots, |x_{d}|)), 
	$$
	$(x_{1}, \ldots, x_{d}) \in \mathbb{R}^{d}$.
	Inversely, a proper copula $C$ on $[0,1]^{d}$ defines the following  L\'{e}vy copula:
	$$
	F(-\epsilon_1  \log(1-u_1), \ldots,-\epsilon_d \log(1-u_d))=-(\prod_{i=1}^{d} \epsilon_i) \log(1-C(u_1, \ldots,u_d)),
	$$
	where $\epsilon_i \in \{-1,1\}$, $i \in \{1, \ldots,d\}$, $(u_{1}, \ldots, u_{d}) \in [0,1]^{d}$.
\end{theorem}

\begin{proof} 
We will establish first a one-to-one correspondence between the probabilistic L\'{e}vy copula and a proper copula and in a second time, we will  establish the result based on the relationship between the probabilistic copula  and the L\'{e}vy copula (equation \ref{eqsec411}).
\begin{itemize}
\item[(step 1)] A proper copula $C$ defines the following probabilistic L\'{e}vy copula:
	$$
	F^{+}( -\log(1-u_1), \ldots,- \log(1-u_d))=- \log(1-C(u_1, \ldots,u_d)),
	$$
due to the fact that $s \mapsto -\log(1-s)$ is absolutely monotone and $C$ is grounded. Inversely, any probabilistic L\'{e}vy copula $F^{+}$ defines a L\'{e}vy measure $\mu$ such that
 $$F^{+}(x_1, \ldots, x_{d})=\mu([1/x_1,+\infty]\times \cdots \times[1/x_d,+\infty] )$$
   and 
$$(x_1, \ldots, x_d)\mapsto 1-\exp(-\mu([1/x_1,+\infty]\times \cdots \times[1/x_d,+\infty] ))$$
 is a distribution function on $[0,+\infty]^d$ (the distribution of the upper record). This function is $d$-increasing and has margins $x_i \mapsto 1-\exp(-x_i)$, $i=1,\ldots,d$. As a consequence, we obtain the following proper copula $C$:
	$$
	C(1-\exp(-x_1), \ldots,1- \exp(-x_d))=1-\exp(-F^{+}(x_1, \ldots, x_{d})).
	$$
\item[(step 2)] The traditional L\'{e}vy copula $F$ is constructed by using the relationship in equation \ref{eqsec411}) and replacing $-sign(x_i)\log(1-G_{i,\epsilon_i}(x_i^{-sign(x_i)}))$ by $y_i$, $i \in \{1, \ldots,d\}$:
\begin{align*}
F(y_1, \ldots,y_d)=-\prod_{i=1}^{d}sgn(y_i)\log\{1-C(1-\exp(-|y_1|,\ldots, 1-\exp(-|y_1|))\}.
\end{align*}
\end{itemize}
\end{proof}

\subsection{Examples of correspondence between  L\'{e}vy copulas and  proper copulas}
Let us illustrate the correspondence with some well-known L\'{e}vy copulas.

\begin{example} \citet[Example  4.2]{barndorff2007levy} discussed the probabilistic L\'{e}vy copula
$$
F^{+}(x_1, \ldots, x_d)=\log(1+(\sum_{i=1}^{d}\frac{\exp(-x_i)}{1-\exp(-x_i)})^{-1}), x_i\geq 0, i  \in \{1, \ldots,d\}.
$$
This function can be put in the following equivalent form:
$$
F^{+}(x_1, \ldots, x_d)=-\log(1-(1+(\sum_{i=1}^{d}\frac{\exp(-x_i)}{1-\exp(-x_i)}))^{-1}), x_i\geq 0, i  \in \{1, \ldots,d\}.
$$
Let us consider the Clayton copula
$$C(u_1,\ldots,u_d)=(1+\sum_{i=1}^{d}(u_i^{-1}-1))^{-1}, u_i \in  [0,1],  i \in \{1, \ldots,d\}.
$$
Let us consider the following distortion of the Clayton copula
\begin{align*}
F^{+}_1(x_1, \ldots, x_d)&=-\log(1-C(1-\exp(-x_1), \ldots, 1-\exp(-x_d)))\\
&=-\log(1-\frac{1}{1+\sum_{i=1}^{d}((1-\exp(-x_i))^{-1}-1)})\\
&=-\log(1-\frac{1}{1+\sum_{i=1}^{d}(\frac{1-(1-\exp(-x_i))}{1-\exp(-x_i)})})\\
&=-\log(1-\frac{1}{1+\sum_{i=1}^{d}(\frac{\exp(-x_i))}{1-\exp(-x_i)})}).
\end{align*}
\end{example}

Any other proper copula can be mapped into a L\'{e}vy copula. Most popular examples of proper copulas are the complete dependence copula 
$$
C_{min}(u_1, \ldots, u_{d})=\min(u_1, \ldots, u_{d}), (u_1, \ldots, u_{d}) \in [0,1]^{d}
$$
and the independence copula
$$
C_{ind}(u_1, \ldots, u_{d})=\prod_{i=1}^{d}u_i , (u_1, \ldots, u_{d}) \in [0,1]^{d}.
$$
Any proper copula $C$ has bounds called Fr\'{e}chet bounds:
$$
\max(0, \sum_{i=1}^{d}u_i-d+1) \leq C(u_1, \ldots, u_{d}) \leq \min(u_1, \ldots, u_{d}), (u_1, \ldots, u_{d}) \in [0,1]^{d}.
$$

\begin{lemma}[Independence and Fr\'{e}chet bounds for L\'{e}vy copulas]
	The complete dependence  and the independence Lévy copulas are respectively
	$$
	F(x_1, \ldots, x_{d})=\min(|x_1|, \ldots, |x_d| ) \prod_{i=1}^{d} sgn(x_i),
	$$
$sgn(x_i)=1$ if $x_{i} \geq 0$, $sgn(x_i)=-1$ if $x_{i}< 0$, $i=1, \ldots,d$	 and 
	$$
	F(x_1, \ldots, x_{d})=\sum_{i=1}^{d} x_{i}\prod_{k \neq i} 1_{x_k=\infty}.
	$$
	For any Lévy copula $F$, we have
	$$
	\max(0,-\log(\sum_{i=1}^{d}\exp(-|x_i|))) \leq F(|x_1|, \ldots, |x_{d}|) \leq \min(|x_1|, \ldots, |x_d| ).
	$$
\end{lemma}

\begin{proof}
	The Lévy copula $F$ defines a copula $C$ by the identity 
	$$
	F(-\log(1-u_1), \ldots,-\log(1-u_d))=- \log(1-C(u_1, \ldots,u_d)),
	$$
	for $(u_{1}, \ldots, u_{d}) \in [0,1]^{d}$. The function $y \in [0,1] \mapsto - \log(1-y)$  is monotone increasing and maps the inequalities 
	$$
	\max(0, \sum_{i=1}^{d}u_i-d+1) \leq C(u_1, \ldots, u_{d}) \leq \min(u_1, \ldots, u_d)
	$$
	into 
	$$
	\max(0,  - \log(d-\sum_{i=1}^{d}u_i))\leq - \log(1-C(u_1, \ldots,u_d)) \leq \min( - \log(1-u_1), \ldots,  - \log(1-u_d));
	$$
	now take $u_i=1-\exp(-|x_{i}|)$, $i=1, \ldots, d$, and we have the result.
\end{proof}

\section{Archimedean L\'{e}vy copulas}
Let us recall that a proper Archimedean copula is a parametric copula 	
	$$
	u=(u_1, \ldots, u_d) \mapsto C_{\psi}(u)=\psi(\sum_{i=1}^{d}\psi^{-1}(u_{i}) ), (u_1, \ldots, u_d) \in [0,1]^{d}
	$$
characterized by a function $\psi: [0, +\infty] \rightarrow [0, 1]$  that verifies: $\psi(0)=1$, $\psi(+\infty)=0$, $(-1)^{j}\psi^{(j)} \geq 0$, $j=1, \ldots, d-2$, $(-1)^{d-2}\psi^{(d-2)}$ decreasing and convex. Similarly,
an Archimedean L\'{e}vy copula $F_{\phi}$ is a function defined as follows:
$$
F_{\phi}(x)=\phi(\sum_{i=1}^{d}\phi^{-1}(x_{i}) ), x=(x_1, \ldots, x_d) \in [0,+\infty]^{d},
$$
where $\phi: [0, +\infty] \rightarrow [0, +\infty]$ verifies $\lim_{t \rightarrow 0}\phi(t)=+\infty$, $\lim_{t \rightarrow +\infty}\phi(t)=0$, $(-1)^{j}\phi^{(j)} \geq 0$, $j=1, \ldots, d-2$, $(-1)^{d-2}\phi^{(d-2)}$ decreasing and convex. Here is how to transform a proper Archimedean copula into an Archimedean L\'{e}vy copula and vice-versa.

\begin{lemma}
If $C_{\psi}$ is a proper Archimedean copula with generator $\psi$, it defines an Archimedean L\'{e}vy  copula $F_{\phi}$ in the following way: take $\phi:[0, +\infty]\rightarrow [0,+\infty]$ such that $\phi(x)= -\log(1-\psi(x)), x \geq 0.$ Then , the function
	$$
	x \mapsto F_{\phi}(x)=\phi(\sum_{i=1}^{d}\phi^{-1}(x_{i}) ), x=(x_1, \ldots, x_d) \in [0,+\infty]^{d}
	$$
	is an Archimedean L\'{e}vy copula on $[0,+\infty)^{d}$. Inversely, any Archimedean L\'{e}vy copula $F_{\phi}$ defines a proper Archimedean copula $C_{\psi}$ with generator $\psi=1-\exp(-\phi)$.
\end{lemma}
\begin{proof}
Let us take  $\psi: [0, +\infty] \rightarrow [0, 1]$ such that $\psi(0)=1$, $\psi(+\infty)=0$, $(-1)^{j}\psi^{(j)} \geq 0$, $j=1, \ldots, d-2$, $(-1)^{d-2}\psi^{(d-2)}$ decreasing and convex; the function $C_{\psi}$ defined by 
	$$
	C_{\psi}(u)=\psi(\sum_{i=1}^{d}\psi^{-1}(u_{i}) ), u=(u_1, \ldots, u_d) \in [0,1]^{d}, 
	$$
is an Archimedean copula.  Let us defines a function $F_{\phi}$ by the formula
\begin{equation}
\label{eqsect51}
 F_{\phi}(x_1, \ldots, x_{d})=-\log(1-C_{\psi}(1-\exp(- x_1), \ldots, 1-\exp(- x_d))).
\end{equation}
Since, for any $x_i>0$, $\psi(x_i)$ is defined as $1-\exp(-\phi(x_i))$, then we have
\begin{align*}
\psi(x_i)=&1-\exp(-\phi(x_i))\\
\psi^{-1}(\psi(x_i))=& \psi^{-1}(1-\exp(-\phi(x_i)))\\
x_i=& \psi^{-1}(1-\exp(-\phi(x_i))) \text{ if } \psi^{-1}\circ\psi=id. 
\end{align*}
If $\psi^{-1}\circ\psi=\psi\circ\psi^{-1}=id$ then we have $\phi^{-1}\circ\phi=\phi\circ\phi^{-1}=id$; by replacing $x_i$ by $\phi^{-1}(x_i)$ in the last equation, we have
\begin{align*}
\phi^{-1}(x_i)=& \psi^{-1}(1-\exp(-\phi\circ\phi^{-1}(x_i))) \\
\phi^{-1}(x_i)=& \psi^{-1}(1-\exp(-x_i)),
\end{align*}
$i=1,\ldots,d$. Here is the expression of the equation \ref{eqsect51} when we replace $\psi^{-1}(1-\exp(-x_i))$ by $\phi^{-1}(x_i)$:
\begin{align*}
 F_{\phi}(x_1, \ldots, x_{d})&=-\log(1-C_{\psi}(1-\exp(- x_1), \ldots, 1-\exp(- x_d)))\\
&=-\log(1-\psi(\sum_{i=1}^{d}\psi^{-1}(1-\exp(- x_i))))\\ 
&=-\log(1-\psi(\sum_{i=1}^{d}\phi^{-1}(x_i)))\\ 
&=-\log(\exp(-\phi(\sum_{i=1}^{d}\phi^{-1}(x_i)))) \text{ by replacing } \psi \text{  by  } 1-\exp(-\phi)\\ 
&=\phi(\sum_{i=1}^{d}\phi^{-1}(x_i)))).
\end{align*}
$F_{\phi}$ is an Archimedean L\'{e}vy copula since $C_{\psi}$ is a copula and the function $s \mapsto -\log(1-s)$ is absolutely monotone. Inversely, given a L\'{e}vy copula $F_{\phi}$  with generator $\phi$, let us consider the transformation $\psi=1-exp(-\phi)$; we have $\phi=-\log(1-\psi)$, $\phi^{-1}(x_i)=\psi^{-1}(1-\exp(-x_i))$, $i \in \{1,\ldots,d\}$,
\begin{align*}
 F_{\phi}(x_1, \ldots, x_{d})&=\phi(\sum_{i=1}^{d}\phi^{-1}(x_i))\\
 &=-\log(1-\psi(\sum_{i=1}^{d}\phi^{-1}(x_i)))\\
  &=-\log(1-\psi(\sum_{i=1}^{d}\psi^{-1}(1-\exp(-x_i))))\\
  &= -\log(1-C_{\psi}(1-\exp(- x_1), \ldots, 1-\exp(- x_d)))
\end{align*}
with $C_{\psi}$ defined by the expression $C_{\psi}(u)=\psi(\sum_{i=1}^{d}\psi^{-1}(u_i))))$, $u=(u_1, \ldots,u_d) \in [0,1]^{d}$. We have not yet proved that $C_{\psi}$ is a proper Archimedean copula; we will do it in corollary \ref{sect5co1} by proving that the generator $\psi$ is a Williamson $d-$transform.
\end{proof}
The generators $\psi \in \Psi_{d}$ of proper Archimedean copulas are $d-$completely monotone functions or equivalently the Williamson $d$-transform of distribution functions. If $X = (X_1, \ldots, X_d)$ is a random vector with the dependence expressed by the Archimedean copula $(u_1,\ldots,u_d)\mapsto \psi(\sum_{i=1}^{n}\psi^{-1}(u_{i})$, then $X$ follows the decomposition $X=R \times (S_1,\ldots,S_{d})$ where $R>0$ and $(S_1, \ldots, S_{d})$ are independent, $S=(S_1,\ldots,S_{d})$ is a uniform random variable on the unit simplex $\{(s_{1},\ldots,s_{d})\in [0,1]^{d}: s_{1}+\cdots+s_{d}=1\}$. Let us call $F_{R}$ the cdf of the radial part.
\begin{align*}
P(X_1>x_1, \ldots,X_d>x_d)&=\int P(S_1>x_1/r, \ldots,S_d>x_d/r)dF_{R}(r)\\
&=\int \max(0, 1- \frac{x_1+\cdots+x_d}{r})^{d-1}dF_{R}(r)\\
&=\mathcal{W}_{d}F_{R}(x_1+\cdots+x_d)\\
&=\psi(x_1+\cdots+x_d).
\end{align*}
More details can be found in \citet{joe2014dependence} and \citet{mcneil2009multivariate}. The generators $\phi \in \Phi_{d}$ of Archimedean L\'{e}vy  copulas has a similar property: they are the Williamson $d$-transform of a logarithmic transform of a cdf. Let us suppose that instead of one observation $X$, we have a sequence of i.i.d. observations $X_n=R_n \times S$ from a L\'{e}vy process; we will suppose that the observation $n$ is a record if and only if $R_{n}>R_{k}$, $k=1, \ldots, n-1$,$n>1$. Let us call $\tilde{N}(A)$ the number of records that fall in a measurable set $A$ and $A_{n}(x_1,\ldots,x_d)$ the event that the $n$th observation is a record and that it falls in the set $[x_1,+\infty]\times \cdots \times [x_d,+\infty]$.
\begin{lemma}
\label{lemma}
Let us suppose that the distribution of the radial part $F_{R}$ is continuous. We have the following results: for $n \geq 1$,
\begin{align*}
P(A_{n}(x_1,\ldots,x_d))&=\int \max(0, 1- \frac{x_1+\cdots+x_d}{r})^{d-1}F_{R}(r)^{n-1} dF_{R}(r);\\
E(\tilde{N}([x_1,+\infty]\times \cdots \times [x_d,+\infty])&=\int \max(0, 1- \frac{x_1+\cdots+x_d}{r})^{d-1} \frac{dF_{R}(r)}{1-F_{R}(r)}.
\end{align*}
\end{lemma}
\begin{proof}
The number of records that fall in the set $[x_1,+\infty]\times \cdots \times [x_d,+\infty]$ is 
$$
\tilde{N}([x_1,+\infty]\times \cdots \times [x_d,+\infty])=\sum_{i=1}^{\infty} 1_{A_{n}(x_1,\ldots,x_d)}.
$$
As a result, we have 
$$
E(\tilde{N}([x_1,+\infty]\times \cdots \times [x_d,+\infty])=\sum_{i=1}^{\infty}E(1_{A_{n}(x_1,\ldots,x_d)})=\sum_{i=1}^{\infty}P(A_{n}(x_1,\ldots,x_d)).
$$
Now, let us remark that $P(A_{n}(x_1,\ldots,x_d))$ can be decomposed into a mixture:
$$
P(A_{n}(x_1,\ldots,x_d))=E(P(A_{n}(x_1,\ldots,x_d)|R_n))=\int P(A_{n}(x_1,\ldots,x_d)|R_n=r)dF_{R}(r).
$$
The quantity $P(A_{n}(x_1,\ldots,x_d)|R_n=r)$ of the mixture has tthe following expression:
\begin{align*}
P(A_{n}(x_1,\ldots,x_d)|R_n=r)&=P(S_1 \geq x_1/r, \ldots, S_d \geq x_d/r, \max_{i \leq n-1} R_i < r )\\
&= \max(0, 1- \frac{x_1+\cdots+x_d}{r})^{d-1}F_{R}(r)^{n-1}.
\end{align*}
Putting the pieces together, and using the fact that $1+p+p^2+\cdots=1/(1-p)$ for $p \in [0,1]$ (with the convention $1/0=\infty$), we have 
\begin{align*}
P(A_{n}(x_1,\ldots,x_d))&=\int \max(0, 1- \frac{x_1+\cdots+x_d}{r})^{d-1}F_{R}(r)^{n-1} dF_{R}(r);\\
E(\tilde{N}([x_1,+\infty]\times \cdots \times [x_d,+\infty])&=\int \max(0, 1- \frac{x_1+\cdots+x_d}{r})^{d-1} \frac{dF_{R}(r)}{1-F_{R}(r)}.
\end{align*}
\end{proof}
The important consequence of lemma \ref{lemma} is that generators of  Archimedean L\'{e}vy copulas are the Williamson $d$-transforms of  hazard measures. Let us suppose that the observations $X_1, X_2, \ldots$ are the realizations of jumps of a L\'{e}vy process with L\'{e}vy measure $\nu$.
\begin{corollary}
\label{sect5co1}
Generators of Archimedean L\'{e}vy copulas are functions of the class		$\Phi_{d}$,
	\begin{align*}
		\Phi_{d}=\{\phi: [0, +\infty] \rightarrow [0,+\infty]: \phi(0)=+\infty, \phi(+\infty)=0,& (-1)^{j}\phi^{(j)}\geq 0,  j \leq d-2, \text{ and } \\
		&(-1)^{d-2}\phi^{(d-2)} \text{ decreasing and convex} \}
	\end{align*}
and $\Phi_{d}$ consists of the Williamson $d-$transforms of functions $H=-\log(1-F_{R})$, $F_{R}$ cdf on $[0,+\infty)$. Moreover, if $\phi$ is a generator of a L\'{e}vy copula, the function $(x_1, \ldots, x_d) \mapsto 1-\exp(-\phi(1/x_1+\cdots+1/x_d))$ defines a distribution on the sigma-algebra generated by the filters on $[0,+\infty]^{d}$ and $1-\exp(-\phi)$ is a Williamson $d-$transform.
\end{corollary}
\begin{proof}
The probability of that no record  falls in a set $[x_1,+\infty]\times \cdots \times [x_d,+\infty]$ is given by the formula
$$
P(\tilde{N}([x_1,+\infty]\times \cdots \times [x_d,+\infty])=0)=\exp(-\nu([x_1,+\infty]\times \cdots \times [x_d,+\infty]))
$$
where $\nu$ is the L\'{e}vy measure because no observation can ever fall in the set $[x_1,+\infty]\times \cdots \times [x_d,+\infty]$ without implying that at  least one record is element of the set. But, since the same probability has a second expression
$$
P(\tilde{N}([x_1,+\infty]\times \cdots \times [x_d,+\infty])=0)=\exp(-E[\tilde{N}([x_1,+\infty]\times \cdots \times [x_d,+\infty])]),
$$
that means that the mean measure of the process $\tilde{N}$  and $\nu$ coincide on sets of the form $[x_1,+\infty]\times \cdots \times [x_d,+\infty]$ and 
\begin{align*}
\nu([x_1,+\infty]\times \cdots \times [x_d,+\infty])&=\int \max(0, 1- \frac{x_1+\cdots+x_d}{r})^{d-1} \frac{dF_{R}(r)}{1-F_{R}(r)}.
\end{align*}
Let us define $\phi$ as follows:
\begin{align*}
\phi(x_1+\cdots+x_d)&=\nu([x_1,+\infty]\times \cdots \times [x_d,+\infty])\\
&=\int \max(0, 1- \frac{x_1+\cdots+x_d}{r})^{d-1} \frac{dF_{R}(r)}{1-F_{R}(r)}.
\end{align*}
The univariate tail integral functions are $x_i \mapsto \phi(x_i)$, $i=1, \ldots,d$, and the L\'{e}vy copula corresponding to the tail integral function is $(x_1, \ldots,x_d)\mapsto \phi(\sum_{i=1}^{d} \phi^{-1}(x_i))$. The distribution of the upper record applied to sets $[1/x_1, +\infty]\times \cdots \times [1/x_d, +\infty]$ gives the function $(x_1, \ldots, x_d) \mapsto 1-\exp(-\phi(1/x_1+\cdots+1/x_d))$ which is a distribution (see theorem \ref{theorem1}). This implies that $1-\exp(-\phi)$ is a Williamson $d-$transform (see \citet{mcneil2009multivariate}).
\end{proof}

\section{Concluding remarks}
We investigate correspondences  between L\'{e}vy copulas and proper copulas. It appears that max-infinite divisible distributions are the missing link that helps to establish the connection. However, the relationship is not straightforward: we have to define a partial order compatible with the inclusion of sets bounded away from the origin and have to consider a definition of max-infinite divisibility that differs from the conventional one. We illustrate the correspondence with the relationship between  proper Archimedean copulas and L\'{e}vy Archimedean copulas. 

\section*{Acknowledgements} Many thanks to Steve Matthews and Fran\c{c}ois Verret from Times Series Research and Analysis Centre (TSRAC-Statistics Canada), Marianna Morano, Jean-Fran\c{c}ois Dubois and Kim Bornais from Special Surveys-Statistics Canada for their support. Many thanks  also to Jean-Marc Fillon and Wesley Yung from Economic Statistical Methods Division-Statistics Canada for their availability and support.
\bibliography{Proposal}
\end{document}